\documentclass[11pt]{article}
\usepackage{hyperref}
\usepackage{times}  
\usepackage{mathpazo}
\usepackage{amssymb,amsmath,amsthm}
\usepackage{epsfig}

\newtheorem{theorem}{Theorem}[section]
\newtheorem{proposition}[theorem]{Proposition}
\newtheorem{definition}[theorem]{Definition}

\newtheorem{claim}[theorem]{Claim}
\newtheorem{lemma}[theorem]{Lemma}

\newtheorem{corollary}[theorem]{Corollary}

\newtheorem{remark}[theorem]{Remark}

\newcommand{\qedsymb}{\hfill{\rule{2mm}{2mm}}}
\renewenvironment{proof}[1][]{\begin{trivlist}
\item[\hspace{\labelsep}{\bf\noindent Proof#1:\/}] }{\qedsymb\end{trivlist}}

\def\calC{{\cal C}}
\def\calD{{\cal D}}
\def\calA{{\cal A}}
\def\calB{{\cal B}}

\def\calF{{\cal F}}

\def\calT{{\cal T}}
\def\calH{{\cal H}}
\def\calS{{\cal S}}
\def\calI{{\cal I}}
\def\Z{{\mathbb{Z}}}

\def\N{\mathbb{N}}
\def\mod{\mbox{mod}}

\newcommand{\eps}{\epsilon}
\renewcommand{\epsilon}{\varepsilon}

\begin{document}

\title{{\bf Sum-free Sets of Integers with a Forbidden Sum}}

\author{
Ishay Haviv\thanks{School of Computer Science, The Academic College of Tel Aviv-Yaffo, Tel Aviv 61083, Israel.
}
}

\date{}

\maketitle

\begin{abstract}
A set of integers is {\em sum-free} if it contains no solution to the equation $x+y=z$.
We study sum-free subsets of the set of integers $[n]=\{1,\ldots,n\}$ for which the integer $2n+1$ cannot be represented as a sum of their elements.
We prove a bound of $O(2^{n/3})$ on the number of these sets, which matches, up to a multiplicative constant, the lower bound obtained by considering all subsets of $B_n = \{ \lceil \frac{2}{3}(n+1) \rceil, \ldots, n \}$.
A main ingredient in the proof is a stability theorem saying that if a subset of $[n]$ of size close to $|B_n|$ contains only a few subsets that contradict the sum-freeness or the forbidden sum, then it is almost contained in $B_n$.
Our results are motivated by the question of counting symmetric complete sum-free subsets of cyclic groups of prime order.
The proofs involve Freiman's $3k-4$ theorem, Green's arithmetic removal lemma, and structural results on independent sets in hypergraphs.

\end{abstract}

\section{Introduction}

For an abelian additive group $G$, a set $A \subseteq G$ is {\em sum-free} if there are no $x,y,z \in A$ such that $x+y=z$.
The study of sum-free sets was initiated in 1916 by Schur~\cite{Schur1916}, who proved that the set of nonzero integers cannot be partitioned into a finite number of sum-free sets.
His work was originally motivated by an attempt to prove the famous Fermat's Last Theorem, which states that the set of all $k$th powers of nonzero integers is sum-free for every $k \geq 3$. To date, over a century later, sum-free sets play a fundamental role in the area of additive combinatorics and enjoy an intensive and fruitful line of research.

Sum-free subsets of the set of integers $[n] = \{1,\ldots,n\}$ have attracted a significant attention in the literature over the years. It is easy to show that the largest size of a sum-free subset of $[n]$ is $\lceil n/2 \rceil$, attained by the set of odd integers in $[n]$ and by the integer interval $[ \lfloor n/2 \rfloor+1, n]$. Cameron and Erd{\H{o}}s~\cite{CameronErdos90} raised the question of counting the sum-free subsets of $[n]$ and conjectured that there are $O(2^{n/2})$ such sets. Their conjecture was confirmed more than a decade later by Green~\cite{Green04} and by Sapozhenko~\cite{Sapozhenko03} independently. More recently, Alon, Balogh, Morris, and Samotij~\cite{AlonRefine14} proved a refined version of the conjecture, providing a bound of $2^{O(n/m)} \cdot {\binom {\lceil n/2 \rceil}{m}}$ on the number of sum-free subsets of $[n]$ of size $m$ for every $1 \leq m \leq \lceil n/2 \rceil$.
The study of structural characterization of sum-free sets of integers was initiated by Freiman~\cite{Freiman92} who showed, roughly speaking, that every sum-free subset of $[n]$ of density greater than $5/12$ either consists entirely of odd integers or is close to an interval.
Freiman's result was extended to all sum-free subsets of $[n]$ of density greater than $2/5$ in an unpublished work by Deshouillers, Freiman, and S\'{o}s.
However, their characterization does not hold for smaller subsets, and the $2/5$ barrier was handled by a more complicated characterization due to Deshouillers, Freiman, S\'{o}s, and Temkin~\cite{DeshouillersFST99}, that was recently further extended by Tran~\cite{Tran17}.

Another setting of great interest in the study of sum-free sets is the cyclic group $\Z_p$ of prime order $p$. The largest size of a sum-free subset of $\Z_p$ is known to be $\lfloor (p+1)/3 \rfloor$.
An explicit characterization of the sum-free sets that attain the largest size was provided in the late sixties by Yap~\cite{Yap68,Yap70}, Diananda and Yap~\cite{DianandaYap69}, and Rhemtulla and Street~\cite{RhemtullaStreet70}. In 2004, Green~\cite{Green04} proved an essentially tight upper bound of $2^{( \frac{1}{3}+o(1)) \cdot p}$ on the number of sum-free subsets of $\Z_p$.
As for their structure, Deshouillers and Lev~\cite{DLev08} proved, improving on Lev~\cite{Lev06} and Deshouillers and Freiman~\cite{DeshouillersF06}, that every sum-free subset of $\Z_p$ whose density is at least $0.318$ is contained, up to an automorphism, in a bounded-size central interval of $\Z_p$.

For a prime $p$, a set $S \subseteq \Z_p$ is said to be {\em symmetric} if $x \in S$ implies that $-x \in S$, and {\em complete} if every element of $\Z_p \setminus S$ can be represented as a sum of two (not necessarily distinct) elements of $S$. The family of symmetric complete sum-free subsets of $\Z_p$ has been considered in the literature, motivated by several applications, such as the study of regular triangle-free graphs with diameter $2$~\cite{HansonS84}, random sum-free sets of positive integers~\cite{Cameron87b,CalkinC98}, and dioid partitions of the group $\Z_p$~\cite{HavivLdioid17} (see~\cite{Cameron87} for a survey). In a recent work~\cite{HavivLsf17}, a full characterization was provided for the symmetric complete sum-free subsets of $\Z_p$ of size at least $(\frac{1}{3}-c) \cdot p$, where $c>0$ is a universal constant. Somewhat surprisingly, this characterization reduces the challenge of counting the symmetric complete sum-free subsets of $\Z_p$ of a given sufficiently large size to counting certain sets of integers.
As will be shortly explained, this question motivates the study of sum-free sets of integers with a forbidden sum, considered in the current work and described below.

\subsection{Sum-free Sets of Integers with a Forbidden Sum}

Let $n \geq 1$ be an integer. For a set $A \subseteq [n]$ and an integer $k \geq 0$, denote
\[kA = \Big\{ \sum_{i=1}^{k}{a_i} ~\Big{|}~ a_1,\ldots,a_k \in A \Big \},\]
and let $\sum A = \cup_{k \geq 0}(kA)$.
In this work we study sum-free subsets of $[n]$ for which the integer $2n+1$ is a forbidden sum, that is, sets $A \subseteq [n]$ satisfying $A \cap 2A = \emptyset$ and $2n+1 \notin \sum A$.
For example, it is easy to see that the interval $B_n = [ \lceil \frac{2}{3}(n+1) \rceil , n]$ satisfies these properties and has size $|B_n| = \lfloor \frac{1}{3}(n+1) \rfloor$. We start with the extremal question of how large can such a set be, and prove the following tight upper bound.

\begin{theorem}\label{thm:extremalIntro}
For every integer $n$, every sum-free set $A \subseteq [n]$ such that $2n+1 \notin \sum A$ satisfies
\[|A| \leq \Big\lfloor \frac{1}{3}(n+1) \Big \rfloor.\]
\end{theorem}
\noindent
In fact, we prove a stronger statement than that of Theorem~\ref{thm:extremalIntro}, providing the same bound under the weaker assumption $2n+1 \notin (3A) \cup (4A)$ rather than $2n+1 \notin \sum A$. This is tight in the sense that the bound is no longer true if we only require $2n+1 \notin 3A$ (see Section~\ref{sec:3A}).

Equipped with the tight answer to the extremal question, we turn to provide a corresponding robust stability theorem, which roughly speaking says the following:
If a subset of $[n]$ of size close to $|B_n|$ contains only a few subsets that contradict the sum-freeness or the forbidden sum, then it is almost contained in $B_n$.
To state the result precisely, let us introduce the following notation.
For an integer $n$, let $\calF^{(3)}_n$ denote the collection of all sets $\{x,y,z\} \subseteq [n]$ of distinct $x,y,z$ satisfying $x+y=z$ or $x+y+z=2n+1$.
For $k \geq 4$, let $\calF^{(k)}_n$ denote the collection of all sets $\{x_1,\ldots,x_k\} \subseteq [n]$ of distinct $x_1,\ldots,x_k$ satisfying $\sum_{i=1}^{k}{x_i}=2n+1$. Our stability theorem is the following.

\begin{theorem}\label{thm:stabilityIntro}
For every $\eps >0$ there exists $\delta = \delta(\eps) > 0$ such that for every sufficiently large integer $n$ the following holds.
Every set $A \subseteq [n]$ of size $|A| \geq (\frac{1}{3}-\delta) \cdot n$ contains at least $\delta \cdot n^{k-1}$ sets from $\calF^{(k)}_n$ for some $k \in \{3,4,5\}$ or satisfies $|A \setminus B_n| \leq \eps n$.
\end{theorem}
\noindent
The proof of Theorem~\ref{thm:stabilityIntro} employs the celebrated Freiman's $3k-4$ theorem~\cite{Freiman59} as well as Green's arithmetic removal lemma~\cite{Green05}.
We remark that, in contrast to Theorem~\ref{thm:extremalIntro}, it is essentially unavoidable to involve the sets of $\calF^{(5)}_n$ in the statement of Theorem~\ref{thm:stabilityIntro}. To see this, consider the set of odd integers in the interval $[1, \lfloor \frac{2n}{3} \rfloor]$ and observe that it is disjoint from $B_n$, contains no set of $\calF^{(3)}_n \cup \calF^{(4)}_n$, and yet has size $|B_n|$.

We finally turn to the question of counting the sum-free sets $A \subseteq [n]$ with $2n+1 \notin \sum A$. Considering all subsets of $B_n$ easily yields a lower bound of $2^{\lfloor \frac{1}{3}(n+1) \rfloor}$. We prove that this is tight up to a multiplicative constant.

\begin{theorem}\label{thm:countingIntro}
There are $\Theta(2^{n/3})$ sum-free sets $A \subseteq [n]$ satisfying $2n+1 \notin \sum A$.
\end{theorem}

The proof of the upper bound in Theorem~\ref{thm:countingIntro} involves two main components.
The first is the study of structural characterization of independent sets in hypergraphs developed by
Saxton and Thomason~\cite{SaxtonT2015} and by Balogh, Morris, and Samotij~\cite{BaloghMS15}, building on a technique of Kleitman and Winston~\cite{KleitmanW82} (see also~\cite{Samotij15}).
We employ a general theorem of~\cite{BaloghMS15} transferring stability results to bounds on the number of independent sets in hypergraphs. This allows us to derive from Theorem~\ref{thm:stabilityIntro} that most sum-free subsets of $[n]$ with $2n+1$ as a forbidden sum are almost contained in the set $B_n$. The second component of the proof, used to count these sets, is an $O(2^{2n/3})$ bound on the number of sets $A \subseteq [n]$ for which $n+1 \notin 3A$ (see Theorem~\ref{thm:Grn/2}; Note that the sum-freeness constraint is not considered here). This bound, which might be of independent interest, is tight up to a multiplicative constant, as follows by considering all subsets of the interval $[\lceil \frac{1}{3}(n+2) \rceil,n]$. Its proof is inspired by a counting technique due to Alon et al.~\cite{AlonRefine14}, and uses Janson's inequality and a bound of Green and Morris~\cite{GreenM16} on the number of sets of integers with small sumset.

Although the current work focuses on the forbidden sum $2n+1$, we remark that Theorem~\ref{thm:countingIntro} can be extended to other forbidden sums around $2n$ as well (see Section~\ref{sec:other_forbidden}).
However, for even forbidden sums in this regime, the situation is somewhat different in the sense that a similar bound of $2^{(\frac{1}{3} + o(1)) \cdot n}$ holds even without assuming the sum-freeness of the sets (see Proposition~\ref{prop:2n}).

\subsection{Symmetric Complete Sum-free Sets in Cyclic Groups}

For a prime $p$, we consider the family of symmetric complete sum-free subsets of the cyclic group $\Z_p$.
It is easy to deduce from the classification results of~\cite{Yap68,Yap70,DianandaYap69,RhemtullaStreet70} that the largest possible size of such a set is $\lfloor (p+1)/3 \rfloor$, attained uniquely, up to an automorphism, by the set $[k+1,2k+1]$ for $p=3k+2$ and by the set $\{k\} \cup [k+2,2k-1] \cup \{2k+1\}$ for $p=3k+1$ where $k \geq 4$.

In a recent work~\cite{HavivLsf17}, the characterization of symmetric complete sum-free subsets of $\Z_p$ of largest size was extended to a linear range of sizes. It was shown there, using a structural result of~\cite{DLev08}, that for all sufficiently large primes $p$, every symmetric complete sum-free subset of $\Z_p$ of (even) size $s \in [0.318 p, \frac{p-1}{3}]$ is, up to an automorphism, of the form
\[ S_T = [p-2s+1,2s-1] \cup \pm (s+T),\]
where $T \subseteq [0,2t-1]$ is a set of $t$ integers for $t = (p-3s+1)/2$. While $S_T$ is symmetric for every set of integers $T$, sufficient and necessary conditions on $T$ for which $S_T$ is complete and sum-free were provided in~\cite{HavivLsf17}.
These conditions reduce the challenge of counting the symmetric complete sum-free subsets of $\Z_p$ of a given sufficiently large size to a question of counting certain sets of integers. As an application of our Theorem~\ref{thm:countingIntro}, we make a step towards this challenge and provide a tight estimation for the number of symmetric complete sum-free sets $S_T$ of size $s$ that satisfy $s \in S_T$ (equivalently, $0 \in T$).

\begin{theorem}\label{thm:S_T_with_sIntro}
For every sufficiently large prime $p$ and every even integer $s \in [0.318p,\frac{p-1}{3}]$, there are $\Theta(2^{(p-3s)/6})$ symmetric complete sum-free sets $S_T \subseteq \Z_p$ of size $s$ that satisfy $s \in S_T$.
\end{theorem}
\noindent
It will be interesting to figure out if a similar upper bound holds for {\em all} symmetric complete sum-free sets $S_T$ of a sufficiently large size $s$ as well.

\paragraph{Outline.}
The rest of the paper is organized as follows.
In Section~\ref{sec:preliminaries}, we gather several definitions and results used throughout the paper.
In Section~\ref{sec:forbidden}, we study sum-free subsets of $[n]$ avoiding the forbidden sum $2n+1$.
Theorem~\ref{thm:extremalIntro} is proved in Section~\ref{sec:max}. For the counting question, we first prove in Section~\ref{sec:super} a weaker upper bound of $2^{(\frac{1}{3} + o(1)) \cdot n}$, and then in Section~\ref{sec:tight} prove our stability result Theorem~\ref{thm:stabilityIntro} and use it to prove the tight answer given in Theorem~\ref{thm:countingIntro}. In Section~\ref{sec:3A} we discuss the tightness of Theorem~\ref{thm:extremalIntro}, and in Section~\ref{sec:other_forbidden} we discuss extensions of Theorem~\ref{thm:countingIntro} to additional forbidden sums.
Finally, in Section~\ref{sec:Z_p}, we present our application to counting symmetric complete sum-free subsets of cyclic groups of prime order and prove Theorem~\ref{thm:S_T_with_sIntro}.

\section{Preliminaries}\label{sec:preliminaries}

\subsection{Additive Combinatorics}\label{sec:additive}

For an abelian additive group $G$ and two sets $A,B \subseteq G$ define $A+B = \{a+b \mid a \in A,~b \in B\}$. We use the notations $k A = \{ \sum_{i=1}^{k}{a_i} \mid a_1,\ldots,a_k \in A\}$ for $k \geq 0$, and $\sum A = \cup_{k \geq 0} (kA)$. The set $A$ is {\em sum-free} if $A \cap 2A = \emptyset$, {\em complete} if $G \setminus A \subseteq 2A$, and {\em symmetric} if $A = -A$.
The following simple claim is well known.
\begin{claim}\label{claim:Freiman}
For every nonempty sets $A,B \subseteq \Z$, $|A+B| \geq |A|+|B|-1$.
In particular, for every $k \geq 0$, $|k A| \geq k \cdot |A|-(k-1)$.
\end{claim}

The following classical theorem, proved in 1959 by Freiman~\cite{Freiman59}, shows that sets of integers with small sumset are highly structured.
\begin{theorem}[Freiman's $3k-4$ Theorem~\cite{Freiman59}]\label{thm:Freiman}
Every finite set $A \subseteq \Z$ satisfying $|A+A| \leq 3|A|-4$ is contained in an arithmetic progression of length at most $|A+A|-|A|+1$.
\end{theorem}

We also need a recent result due to Tran~\cite{Tran17} on the structure of large sum-free subsets of $[n]$ (see also~\cite{DeshouillersFST99}).
\begin{theorem}[\cite{Tran17}]\label{thm:Tran}
There exists a constant $c>0$ such that for every integer $n$ and real $\eta \in [ \frac{2}{n} ,c]$ the following holds.
Every sum-free set $A \subseteq [n]$ of size $|A| \geq (\frac{2}{5}-\eta) \cdot n$ satisfies one of the following alternatives.
\begin{enumerate}
  \item\label{itm:1} All the elements of $A$ are congruent to $1$ or $4$ modulo $5$.
  \item\label{itm:2} All the elements of $A$ are congruent to $2$ or $3$ modulo $5$.
  \item\label{itm:3} All the elements of $A$ are odd.
  \item\label{itm:4} $\min(A) \geq |A|$.
  \item\label{itm:5} $A \subseteq [(\frac{1}{5} - 200\eta^{1/2})n, (\frac{2}{5} + 200\eta^{1/2})n] \cup [(\frac{4}{5}-200\eta^{1/2})n,n]$.
\end{enumerate}
\end{theorem}

\subsection{Green's Arithmetic Removal Lemma}\label{sec:Green}

In 2005, Green~\cite{Green05} proved an arithmetic removal lemma for abelian groups, motivated by well-known removal lemmas in graph theory (see~\cite{KSV12} for an alternative proof and an extension).
Among other applications, he used it to prove that every `almost' sum-free subset of $[n]$ can be made sum-free by removing relatively few elements. We state below the arithmetic removal lemma of~\cite{Green05} and a variant of its application to sum-freeness, a proof of which is included for completeness.

\begin{theorem}[\cite{Green05}]\label{thm:Green}
For every $\eps >0$ and an integer $k \geq 3$ there exists $\delta = \delta_k(\eps) >0$, such that for every integer $N$ the following holds.
Let $A_1, \ldots, A_k$ be subsets of an abelian additive group $G$ of size $|G|=N$ such that the number of zero-sum $k$-tuples in $\prod_{i=1}^{k}{A_i}$ (that is, $k$-tuples $(x_1,\ldots,x_k) \in \prod_{i=1}^{k}{A_i}$ satisfying $\sum_{i=1}^{k}{x_i} = 0$) is at most $\delta \cdot N^{k-1}$. Then, there exist subsets $A'_i \subseteq A_i$, $i \in [k]$, with $|A'_i| \leq \eps \cdot N$, for which there are no zero-sum $k$-tuples in $\prod_{i=1}^{k}{(A_i \setminus A'_i)}$.
\end{theorem}

\begin{corollary}\label{cor:Green_sets}
For every $\eps >0$ there exists $\delta' = \delta'(\eps) >0$, such that for every sufficiently large integer $n$ the following holds.
Let $\calF$ denote the collection of all sets $\{x,y,z\} \subseteq [n]$ of distinct $x,y,z$ satisfying $x+y=z$.
If a set $A \subseteq [n]$ contains at most $\delta' \cdot n^{2}$ sets from $\calF$
then there exists a subset $A' \subseteq A$ of size $|A'| \leq \eps \cdot n$ for which $A \setminus A'$ is sum-free.
\end{corollary}

\begin{proof}
For $\eps >0$ let $\delta = \delta_3 (\frac{\eps}{6})$, where $\delta_3$ is as in Theorem~\ref{thm:Green}, and define $\delta' = \frac{\delta}{2}$.
We apply Theorem~\ref{thm:Green} with the group $G = \Z_{2n}$ and $k=3$.
Identify a set $A \subseteq [n]$ as a subset of $G$ in the natural way, and consider the sets $A_1 = A_2 = A$ and $A_3 = -A$.
Observe that for $x,y,z \in [n]$, the equality $x+y=z$ over $G$ is equivalent to the same equality over the integers.
Assuming that $A$ contains at most $\delta' \cdot n^{2}$ sets from $\calF$, the number of ordered triples $(x,y,z) \in A^3$ such that $x+y=z$ is at most
$3! \cdot \delta' n^2 + 3n \leq 8 \delta' n^2 = \delta \cdot |G|^2$, where the inequality holds assuming that $n$ is sufficiently large.
By Theorem~\ref{thm:Green}, there exists a set $A' \subseteq A$ of size $|A'| \leq 3 \cdot \frac{\eps}{6} \cdot |G| = \eps \cdot n$ such that $A \setminus A'$ is sum-free over the group $G$, thus over the integers as well.
\end{proof}

As another application of Theorem~\ref{thm:Green}, we show that for every fixed $k$, if a subset of $[n]$ includes a relatively few $k$-subsets with a given sum then it has a large subset including no $k$-tuples with this sum at all.

\begin{corollary}\label{cor:3A4A_sets}
For every $\eps >0$ and an integer $k \geq 3$ there exists $\delta'' = \delta''_k(\eps) >0$, such that for every sufficiently large integer $n$ the following holds.
For an integer $\ell$, let $\calF$ denote the collection of all sets $\{x_1,\ldots,x_k\} \subseteq [n]$ of distinct $x_1,\ldots,x_k$ satisfying $\sum_{i=1}^{k}{x_i} = \ell$.
If a set $A \subseteq [n]$ contains at most $\delta'' \cdot n^{k-1}$ sets from $\calF$
then there exists a subset $A' \subseteq A$ of size $|A'| \leq \eps \cdot n$ for which $\ell \notin k(A \setminus A')$.
\end{corollary}

\begin{proof}
For $\eps >0$ and $k \geq 3$ let $\delta'' = \delta_k (\frac{\eps}{k^2})$, where $\delta_k$ is as in Theorem~\ref{thm:Green}.
It can be assumed that $\ell \in [k,kn]$, as otherwise $\ell \notin kA$ for every $A \subseteq [n]$.
We apply Theorem~\ref{thm:Green} with the group $G = \Z_{kn}$ and the integer $k$.
Identify a set $A \subseteq [n]$ as a subset of $G$ in the natural way, and consider the sets $A_1 = \cdots = A_{k-1} = A$ and $A_k = A-\ell$.
Observe that for $x_1,\ldots,x_k \in [n]$, the equality $\sum_{i=1}^{k-1}{x_i}+(x_k-\ell)=0$ can be written as $\sum_{i=1}^{k}{x_i}=\ell$, and that it holds over $G$ if and only if it holds over the integers.
Assuming that $A$ contains at most $\delta'' \cdot n^{k-1}$ sets from $\calF$, the number of ordered $k$-tuples $(x_1,\ldots,x_k) \in A^k$ such that $\sum_{i=1}^{k}{x_i}=\ell$ is at most
\[k! \cdot \delta'' \cdot n^{k-1} + {\binom k 2}n^{k-2} \leq \frac{3}{2} \cdot k! \cdot \delta'' \cdot n^{k-1} = \frac{3k!}{2k^{k-1}} \cdot \delta'' \cdot |G|^{k-1} \leq \delta'' \cdot |G|^{k-1},\]
where the first inequality holds assuming that $n$ is sufficiently large.
By Theorem~\ref{thm:Green}, there exists a set $A' \subseteq A$ of size $|A'| \leq k \cdot \frac{\eps}{k^2} \cdot |G| = \eps \cdot n$ such that $\ell \notin k(A \setminus A')$ over the group $G$, thus over the integers as well.
\end{proof}

\subsection{Independent Sets in Hypergraphs}\label{sec:independent}

Structural results on independent sets in hypergraphs were found in recent years as a strong tool in proving extremal, structural, and counting results in combinatorics (see, e.g.,~\cite{SaxtonT2015,BaloghMS15,Samotij15}).
We state below a theorem of Balogh, Morris, and Samotij~\cite{BaloghMS15} that provides a general framework to derive counting statements from supersaturation and stability results.

We start with a few notations. For a hypergraph $\calH$, denote by $V(\calH)$ the set of its vertices and by $E(\calH) \subseteq P(V(\calH))$ the set of its hyperedges. Let $v(\calH) = |V(\calH)|$ and $e(\calH) = |E(\calH)|$. The hypergraph $\calH$ is $k$-uniform if $|e|=k$ for every $e \in E(\calH)$.
For a set $A \subseteq V(\calH)$ let $\calH[A]$ be the subhypergraph of $\calH$ induced by $A$. An independent set in $\calH$ is a subset of $V(\calH)$ containing no hyperedge of $\calH$. Let $\calI(\calH)$ denote the family of independent sets in $\calH$, and for an integer $m$, let $\calI(\calH,m)$ denote the family of independent sets in $\calH$ of size $m$.
For a set $T \subseteq V(\calH)$ define $\deg_\calH(T) = |\{e \in E(\calH) \mid T \subseteq e\}|$, and for an integer $\ell$, let
\[\Delta_\ell(\calH) = \max \{ \deg_{\calH}(T) \mid T \subseteq V(\calH)~\mbox{ and }~|T|=\ell \}.\]
We also need the following definitions of density and stability of hypergraphs used in~\cite{BaloghMS15} (see also~\cite{AlonBMS14,Schacht16}).

\begin{definition}\label{def:dense_stable}
Let $\calH = (\calH_n)_{n \in \N}$ be a sequence of hypergraphs, and let $\alpha \in (0,1)$ be a real number.
\begin{enumerate}
  \item\label{itm:dense_def} We say that $\calH$ is $\alpha$-{\em dense} if for every $\eps >0$ there exist $\delta>0$ and $n_0$ such that the following holds. For every $n \geq n_0$ and a set $A \subseteq V(\calH_n)$ with $|A| \geq (\alpha+\eps) \cdot v(\calH_n)$, $e(\calH_n[A]) \geq \delta \cdot e(\calH_n)$.
  \item\label{itm:stable_def} For a sequence $\calB$ of sets $\calB_n \subseteq P(V(\calH_n))$, we say that $\calH$ is $(\alpha,\calB)$-{\em stable} if for every $\eps>0$ there exist $\delta >0$ and $n_0$ such that the following holds. For every $n \geq n_0$ and a set $A \subseteq V(\calH_n)$ with $|A| \geq (\alpha-\delta) \cdot v(\calH_n)$, it holds that $e(\calH_n[A]) \geq \delta \cdot e(\calH_n)$ or $|A \setminus B| \leq \eps \cdot v(\calH_n)$ for some $B \in \calB_n$.
\end{enumerate}
\end{definition}

\begin{theorem}[Theorems~5.4 and~6.3 in~\cite{BaloghMS15}]\label{thm:BMS_items}
Let $\calH = (\calH_n)_{n \in \N}$ be a sequence of $k$-uniform hypergraphs for an integer $k$, and let $\alpha \in (0,1)$ and $c>0$. Let $p \in [0,1]^\N$ be a sequence of real numbers satisfying that for every sufficiently large integer $n$ and every $\ell \in [k]$,
\[ \Delta_\ell(\calH_n) \leq c \cdot p_n^{\ell-1} \cdot \frac{e(\calH_n)}{v(\calH_n)}.\]
\begin{enumerate}
  \item\label{itm:BMS_1} If $\calH$ is $\alpha$-dense then for every $\gamma >0$ there exists $C>0$ such that for every sufficiently large $n$ and every $m \geq C p_n v(\calH_n)$,
\[ |\calI(\calH_n,m)| \leq {\binom {(\alpha+\gamma) \cdot v(\calH_n)} {m}}.\]
  \item\label{itm:BMS_2} Let $\calB$ be a sequence of sets $\calB_n \subseteq P(V(\calH_n))$. If $\calH$ is $(\alpha,\calB)$-stable then for every $\gamma >0$ there exist $\beta >0$ and $C>0$ such that for every sufficiently large $n$ and every $m \geq C p_n v(\calH_n)$ there are at most
\[ (1-\beta)^m \cdot {\binom {\alpha \cdot v(\calH_n)} {m}}\]
independent sets $I \in \calI(\calH_n,m)$ such that $|I \setminus B| \geq \gamma m$ for every $B \in \calB_n$.
\end{enumerate}
\end{theorem}

\subsection{Janson's Inequality}

Janson's inequality is a useful tool to bound the probability that no event of a collection of `mostly' independent events occurs. See, e.g.,~\cite[Chapter~8]{AlonS16}.

\begin{lemma}[Janson's Inequality]\label{lemma:Janson}
Let $\{B_i\}_{i \in J}$ be a family of subsets of a finite set $X$ and let $p \in [0,1]$. Denote
\[ \mu = \sum_{i \in J}{p^{|B_i|}}~~~\mbox{and}~~~\Delta = \sum_{i \sim j}{p^{|B_i \cup B_j|}},\]
where $i \sim j$ means that $i$ and $j$ are distinct indices in $J$ satisfying $B_i \cap B_j \neq \emptyset$.
Let $R$ be a random subset of $X$, where every element of $X$ is chosen to be in $R$ independently with probability $p$.
Then, the probability that no $B_i$ for $i \in J$ is contained in $R$ is at most $\max (e^{-\mu/2},e^{-\mu^2/(2\Delta)})$.
\end{lemma}

\subsection{Sets of Integers with Small Sumset}\label{sec:sumset}

We need the following bound due to Green and Morris~\cite{GreenM16} on the number of sets of integers with a bounded-size sumset.
\begin{theorem}[Theorem~1.1 in~\cite{GreenM16}]\label{thm:partitions_GM}
For every $\delta > 0$ and $\lambda > 0$, for a sufficiently large integer $\ell$ the following holds.
For every $k \in \N$ there are at most
\[2^{\delta \ell} \cdot {\binom {\lambda \ell/2} {\ell} } \cdot k^{\lfloor \lambda +\delta \rfloor}\]
sets $S \subseteq [k]$ with $|S|=\ell$ and $|S+S| \leq \lambda \cdot \ell$.
\end{theorem}

We also use, in this context, the following simple bound given in~\cite{AlonRefine14} on the number of integer partitions of an integer $k$ into $\ell$ distinct parts, i.e., the number of sets of $\ell$ positive integers whose sum is $k$.
\begin{lemma}[Lemma~5.1 in~\cite{AlonRefine14}]\label{lemma:partitions_basic}
For every two positive integers $k$ and $\ell$, the number of partitions of $k$ into $\ell$ distinct parts is at most
\[\Big (\frac{e^2k}{\ell^2} \Big )^\ell.\]
\end{lemma}

\section{Sum-free Sets of Integers with a Forbidden Sum}\label{sec:forbidden}

We study the sum-free sets $A \subseteq [n]$ that satisfy $2n+1 \notin \sum A$, where the integer $2n+1$ is referred to as a forbidden sum.
We start with the extremal question of how large can such a set be, and then turn to study the number of these sets.

\subsection{The Maximum Size}\label{sec:max}

The following theorem confirms Theorem~\ref{thm:extremalIntro}.
\begin{theorem}\label{thm:n/3}
For every integer $n$, every sum-free set $A \subseteq [n]$ such that $2n+1 \notin (3A) \cup (4A)$ satisfies
\[|A| \leq \Big\lfloor \frac{1}{3}(n+1) \Big \rfloor.\]
\end{theorem}

\begin{proof}
Let $A \subseteq [n]$ be a sum-free set such that $2n+1 \notin (3A) \cup (4A)$.
Denote \[C = A \cup (A+A) \subseteq [2n].\]
By the sum-freeness of $A$ and Claim~\ref{claim:Freiman},
\[ |C| = |A| + |A + A| \geq |A| + (2|A|-1) = 3|A|-1.\]
We claim that $|C| \leq n$. Otherwise, by the pigeonhole principle, there exists $1 \leq i \leq n$ for which $\{i, 2n+1-i\} \subseteq C$.
Notice that $i$ belongs to either $A$ or $A+A$ and that $2n+1-i$, which is larger than $n$, belongs to $A+A$. This implies that $2n+1 \in (3A) \cup (4A)$, in contradiction to our assumption.
It follows that $3|A|-1 \leq |C| \leq n$, hence $|A| \leq \lfloor \frac{1}{3}(n+1) \rfloor$ as required.
\end{proof}

\begin{remark}
We note that the assumption $2n+1 \notin (3A) \cup (4A)$ in Theorem~\ref{thm:n/3} cannot be relaxed to the assumption $2n+1 \notin 3A$. See Section~\ref{sec:3A} for a detailed discussion.
\end{remark}

To obtain a matching lower bound, consider the interval $B_n = [ \lceil \frac{2}{3}(n+1) \rceil,n]$ whose size is
\[ |B_n| = n - \Big(\Big\lceil \frac{2}{3}(n+1) \Big\rceil-1\Big) = \Big\lfloor \frac{1}{3}(n+1) \Big \rfloor.\]
Clearly, the sum of every two element of $B_n$ is smaller than $2n+1$ and the sum of every three is larger than $2n+1$.
This implies that $2n+1 \notin \sum B_n$ and, in particular, that $2n+1 \notin (3B_n) \cup (4B_n)$. Combining it with Theorem~\ref{thm:n/3}, we derive the following corollary.

\begin{corollary}\label{cor:max_set}
For every integer $n$, the following holds.
\begin{enumerate}
  \item The maximum size of a sum-free set $A \subseteq [n]$ such that $2n+1 \notin \sum A$ is $\lfloor \frac{1}{3}(n+1) \rfloor$.
  \item The maximum size of a sum-free set $A \subseteq [n]$ such that $2n+1 \notin (3A) \cup (4A)$ is $\lfloor \frac{1}{3}(n+1) \rfloor$.
\end{enumerate}
\end{corollary}

\subsection{Supersaturation}\label{sec:super}

We turn to prove a supersaturation result for sum-free subsets of $[n]$ with $2n+1$ as a forbidden sum. Namely, we show that every set $A \subseteq [n]$ of size linearly larger than the bound given in Theorem~\ref{thm:n/3} contains many subsets that contradict the sum-freeness or the forbidden sum. To state it formally, we recall the following notation.

\begin{definition}\label{def:forbidden_sets}
For an integer $n$, let $\calF^{(3)}_n$ denote the collection of all sets $\{x,y,z\} \subseteq [n]$ of distinct $x,y,z$ satisfying $x+y=z$ or $x+y+z=2n+1$.
For $k \geq 4$, let $\calF^{(k)}_n$ denote the collection of all sets $\{x_1,\ldots,x_k\} \subseteq [n]$ of distinct $x_1,\ldots,x_k$ satisfying $\sum_{i=1}^{k}{x_i}=2n+1$.
\end{definition}

The proof of the supersaturation result, stated below, uses the corollaries of Green's arithmetic removal lemma given in Section~\ref{sec:Green}.

\begin{theorem}\label{thm:supersaturation}
For every $\eps >0$ there exists $\delta = \delta(\eps) >0$ such that for every sufficiently large integer $n$ the following holds.
Every set $A \subseteq [n]$ of size $|A| \geq (\frac{1}{3}+\eps) \cdot n$ contains at least $\delta \cdot n^{k-1}$ sets from $\calF_n^{(k)}$ for some $k \in \{3,4\}$.
\end{theorem}

\begin{proof}
For a given $\eps >0$,  define $\delta = \min ( \delta'(\frac{\eps}{6}), \delta''_3(\frac{\eps}{6}), \delta''_4(\frac{\eps}{6}))$, where $\delta'$, $\delta''_3$, and $\delta''_4$ are as in Corollaries~\ref{cor:Green_sets} and~\ref{cor:3A4A_sets}.
Assume by contradiction that for a sufficiently large integer $n$, a set $A \subseteq [n]$ of size $|A| \geq (\frac{1}{3}+\eps) \cdot n$ contains fewer than $\delta \cdot n^{k-1}$ sets from $\calF_n^{(k)}$ for every $k \in \{3,4\}$.
Applying Corollary~\ref{cor:Green_sets}, we obtain a set $D_1 \subseteq A$ of size $|D_1| \leq \frac{\eps}{6} \cdot n$ for which $A \setminus D_1$ is sum-free.
Applying Corollary~\ref{cor:3A4A_sets} twice with $\ell = 2n+1$ and $k\in \{3,4\}$, we obtain two sets $D_2, D_3 \subseteq A$, each of which is of size at most $\frac{\eps}{6} \cdot n$, such that $2n+1 \notin 3(A \setminus D_2)$ and $2n+1 \notin 4(A \setminus D_3)$.
Consider the set $B = A \setminus (D_1 \cup D_2 \cup D_3)$, and notice that $B$ is sum-free and satisfies $2n+1 \notin (3B) \cup (4B)$. The size of $B$ satisfies
\[|B| \geq |A|-|D_1|-|D_2|-|D_3| \geq \Big (\frac{1}{3}+\eps \Big ) \cdot n - 3 \cdot \frac{\eps}{6} \cdot n \geq  \Big (\frac{1}{3}+\frac{\eps}{2} \Big ) \cdot n.\]
We get a contradiction to Theorem~\ref{thm:n/3}, so we are done.
\end{proof}

Theorem~\ref{thm:supersaturation} combined with a result of~\cite{BaloghMS15} on counting independent sets in hypergraphs (see Section~\ref{sec:independent}) already allows us to derive a bound of $2^{(\frac{1}{3}+o(1)) \cdot n}$ on the number of sum-free sets $A \subseteq [n]$ with $2n+1 \notin (3A) \cup (4A)$, and, in particular, on those that satisfy $2n+1 \notin \sum A$.
We include a proof of this bound for didactical reasons and then turn to prove the tighter bound of $O(2^{n/3})$ (see Section~\ref{sec:tight}).

\begin{corollary}
There are $2^{(\frac{1}{3}+o(1)) \cdot n}$ sum-free sets $A \subseteq [n]$ satisfying $2n+1 \notin (3A) \cup (4A)$.
\end{corollary}

\begin{proof}
Let $\calH = (\calH_n)_{n \in \N}$ be a sequence of $4$-uniform hypergraphs defined as follows. For every $n$, $\calH_n$ is the hypergraph on the vertex set $[n]$ whose hyperedges are all $4$-subsets of $[n]$ that contain at least one of the sets in $\calF^{(3)}_n \cup \calF^{(4)}_n$.
Observe that $e(\calH_n) = \Theta(n^3)$, and that $\Delta_1(\calH_n) = \Theta(n^2)$, $\Delta_2(\calH_n) = \Theta(n)$, $\Delta_3(\calH_n) = \Theta(n)$, and $\Delta_4(\calH_n) = 1$.
Hence, for every sufficiently large $n$, every $\ell \in [4]$, some constant $c>0$, and $p_n = 1/\sqrt{n}$, we have
\[ \Delta_\ell(\calH_n) \leq c \cdot p_n^{\ell-1} \cdot \frac{e(\calH_n)}{v(\calH_n)}.\]
Since every sum-free set $A \subseteq [n]$ satisfying $2n+1 \notin (3A) \cup (4A)$ forms an independent set in $\calH_n$, it suffices to bound from above the size of $\calI(\calH_n)$.

To this end, let us show that the sequence of hypergraphs $\calH$ is $\frac{1}{3}$-dense (recall Definition~\ref{def:dense_stable}, Item~\ref{itm:dense_def}).
Let $\eps >0$ be a constant.
By Theorem~\ref{thm:supersaturation}, for some $\delta>0$ and every sufficiently large integer $n$, every set $A \subseteq [n]$ of size $|A| \geq (\frac{1}{3}+\eps) \cdot n$ contains at least $\delta \cdot n^{k-1}$ sets from $\calF^{(k)}_n$ for some $k \in \{3,4\}$.
Observe that this implies, using $e(\calH_n) = \Theta(n^3)$, that such an $A$ satisfies $e(\calH_n[A]) \geq \delta' \cdot e(\calH_n)$ for some $\delta'>0$, as required.

Now, we can apply Item~\ref{itm:BMS_1} of Theorem~\ref{thm:BMS_items} to obtain that for every $\gamma >0$ there exists $C>0$ such that for every sufficiently large $n$ and every $m \geq C p_n v(\calH_n) = C \sqrt{n}$,
\[ |\calI(\calH_n,m)| \leq {\binom {(\frac{1}{3}+\gamma)n} {m}}.\]
It follows that
\[ |\calI(\calH_n)| \leq \sum_{m=0}^{C\sqrt{n}}{{\binom {n}{m}}} + \sum_{m=C\sqrt{n}}^{(\frac{1}{3}+\gamma)n}{ {\binom {(\frac{1}{3}+\gamma)n} {m}} } \leq n^{O(\sqrt{n})}+2^{(\frac{1}{3}+\gamma) \cdot n} \leq 2^{(\frac{1}{3}+2\gamma) \cdot n}.\]
Since the bound holds for every $\gamma >0$, the result follows.
\end{proof}

\subsection{The Tight Bound -- Proof of Theorem~\ref{thm:countingIntro}}\label{sec:tight}

In this section we estimate the number of sum-free sets $A \subseteq [n]$ with $2n+1 \notin \sum A$ and confirm Theorem~\ref{thm:countingIntro}.
Recall that a lower bound of $2^{\lfloor \frac{1}{3}(n+1) \rfloor}$ follows by considering all subsets of the set $B_n = [ \lceil \frac{2}{3}(n+1) \rceil , n]$.
For the upper bound, we prove the following stronger statement.

\begin{theorem}\label{thm:3A4A5A}
There are $O(2^{n/3})$ sum-free sets $A \subseteq [n]$ satisfying $2n+1 \notin (3A) \cup (4A) \cup (5A)$.
\end{theorem}

\paragraph{A roadmap for the proof of Theorem~\ref{thm:3A4A5A}.}
In the proof, we count separately the sets that include `many' elements which do not belong to $B_n$ and the sets that are almost contained in $B_n$.
The former sets are considered in Section~\ref{sec:stability}, where we prove our stability result Theorem~\ref{thm:stabilityIntro}, combine it
with a result of~\cite{BaloghMS15} on counting independent sets in hypergraphs, and obtain the required bound (see Corollary~\ref{cor:far_from_extremal}).
To count the sets that are almost contained in $B_n$, we consider two subcases: The sets whose elements are all greater than $n/2$ are considered in Section~\ref{sec:n+1_3A} (see Corollary~\ref{cor:Larger_n/2}), and those that include at least one smaller element are considered in Section~\ref{sec:easy} (see Lemma~\ref{lemma:o(2^{n/3})}). We finally put everything together and derive Theorem~\ref{thm:3A4A5A} in Section~\ref{sec:final_proof}.

\begin{remark}\label{remark:divisibility}
For simplicity of presentation, we omit throughout this section floor and ceiling signs whenever the implicit assumption that a certain quantity is integer makes no essential difference in the argument.
\end{remark}

\subsubsection{Stability}\label{sec:stability}

We restate and prove our stability result (recall Definition~\ref{def:forbidden_sets}).

\begin{theorem}\label{thm:stability}
For every $\eps >0$ there exists $\delta = \delta(\eps) > 0$ such that for every sufficiently large integer $n$ the following holds.
Every set $A \subseteq [n]$ of size $|A| \geq (\frac{1}{3}-\delta) \cdot n$ contains at least $\delta \cdot n^{k-1}$ sets from $\calF^{(k)}_n$ for some $k \in \{3,4,5\}$ or satisfies $|A \setminus [\frac{2n}{3}+1,n]| \leq \eps n$.
\end{theorem}

We need the following simple lemma.

\begin{lemma}\label{lemma:stability_3A}
For $\eps >0$ and a sufficiently large integer $n$, let $A \subseteq [n]$ be a set of size $|A| \geq ( \frac{2}{3} - \eps) \cdot n$ satisfying either $2n+2 \notin 3A$ or $2n+1 \notin 3A$. Then, $|A \cap [\frac{2n}{3}+1,n]| \leq 3\eps n$.
\end{lemma}

\begin{proof}
We prove the lemma under the assumption $2n+2 \notin 3A$. The case of $2n+1 \notin 3A$ is similar.
Assume by contradiction that $|A \cap [\frac{2n}{3}+1,n]| > 3\eps n$. Denote $\max(A) = \frac{2n}{3}+a$, and notice that $a \in (3\eps n ,\frac{n}{3}]$.
Consider the interval $[\frac{2n}{3}-2a+2,\frac{2n}{3}+a]$ of length $3a-1$, and observe that it contains $\lceil \frac{3a-2}{2} \rceil$ pairwise disjoint sets $\{x,y\}$ of distinct $x,y$ satisfying $x+y = \frac{4n}{3}-a+2$. Since $\frac{2n}{3}+a \in A$ and $2n+2 \notin 3A$, it follows that at least one element from every such set does not belong to $A$, hence
\[ |A| \leq \max(A) - \Big \lceil \frac{3a-2}{2} \Big\rceil = \frac{2n}{3} - \Big\lfloor \frac{a-1}{2} \Big\rfloor < \Big (\frac{2}{3}-\eps \Big ) \cdot n,\]
in contradiction.
\end{proof}

\begin{proof}[ of Theorem~\ref{thm:stability}]
For a given $\eps > 0$ define
\[\delta = \min \Big( \delta' \Big(\frac{\eps}{96} \Big), \delta''_3 \Big(\frac{\eps}{96} \Big), \delta''_4 \Big(\frac{\eps}{96} \Big), \delta''_5 \Big(\frac{\eps}{96} \Big), \frac{\eps}{24} \Big),\]
where $\delta'$, $\delta''_3$, $\delta''_4$, and $\delta''_5$ are as in Corollaries~\ref{cor:Green_sets} and~\ref{cor:3A4A_sets}.
Notice that it can be assumed, whenever needed, that $\eps$ is sufficiently small (because the statement of the theorem is stronger for smaller values of $\eps$).
For a sufficiently large integer $n$, let $A \subseteq [n]$ be a set of size $|A| \geq (\frac{1}{3}-\delta) \cdot n$.
Assume that for each $k \in \{3,4,5\}$, fewer than $\delta \cdot n^{k-1}$ of the sets in $\calF^{(k)}_n$ are contained in $A$.
Our goal is to prove that $|A \setminus [\frac{2n}{3}+1,n]| \leq \eps n$.

We first apply Corollaries~\ref{cor:Green_sets} and~\ref{cor:3A4A_sets}, once and three times respectively, to obtain a set $A' \subseteq A$ such that $A'$ is sum-free and satisfies $2n+1 \notin (3A') \cup (4A') \cup (5A')$ and $|A \setminus A'| \leq 4 \cdot \frac{\eps}{96} \cdot n = \frac{\eps n}{24}$. Observe that
\begin{eqnarray}\label{eq:|A'|}
|A'| = |A| - |A \setminus A'| \geq \Big(\frac{1}{3}-\delta-\frac{\eps}{24} \Big ) \cdot n \geq \Big (\frac{1}{3}-\frac{\eps}{12} \Big) \cdot n,
\end{eqnarray}
and that, by Theorem~\ref{thm:n/3}, we have $|A'| \leq \frac{n}{3}$ (recall Remark~\ref{remark:divisibility}).

Consider the set $C = A' \cup (A'+A') \subseteq [2n]$.
Observe that $|C| \leq n$, as otherwise, by the pigeonhole principal, there exists $1 \leq i \leq n$ for which $\{i, 2n+1-i\} \subseteq C$, in contradiction to $2n+1 \notin (3A') \cup (4A')$.
By the sum-freeness of $A'$ and~\eqref{eq:|A'|}, it follows that
\[ |A'| + |A' + A'| = |C| \leq n \leq \frac{|A'|}{\frac{1}{3}-\frac{\eps}{12}} = \frac{3 |A'|}{1-\frac{\eps}{4}} \leq 3 |A'| \cdot \Big (1+\frac{\eps}{3} \Big).\]
Rearranging, we obtain that
\[|A'+A'| \leq |A'| \cdot (2+\eps ) \leq 3|A'|-4.\]
Hence we can apply Theorem~\ref{thm:Freiman} to obtain that $A'$ is contained in an arithmetic progression of length at most
\[|A'+A'|-|A'|+1 \leq |A'| \cdot (1+\eps )+1 \leq \frac{n}{3} \cdot (1+ \eps )+1.\]
It follows that one can exclude at most $\frac{\eps n}{3} +1$ elements from $A'$ to get a subset $A'' \subseteq A'$ contained in an arithmetic progression of length precisely $\frac{n}{3}$ such that
\begin{eqnarray}\label{eq:|AminusA''|}
|A \setminus A''| \leq |A \setminus A'|+ \frac{\eps n}{3} +1 \leq  \frac{\eps n}{24} + \frac{\eps n}{3} +1 = \frac{3\eps n}{8} +1
\end{eqnarray}
and, using~\eqref{eq:|A'|},
\begin{eqnarray}\label{eq:|A''|}
|A''| = |A| - |A \setminus A''| \geq \Big(\frac{1}{3}-\frac{\eps}{24}\Big) \cdot n - \Big ( \frac{3\eps n}{8} +1 \Big ) \geq \Big(\frac{1}{3}-\frac{\eps}{2} \Big) \cdot n.
\end{eqnarray}
Since $A'' \subseteq A'$, we have $2n+1 \notin (3A'') \cup (4A'') \cup (5A'')$.

Let $B$ denote an arithmetic progression of length $\frac{n}{3}$ containing the set $A''$, and let $d$ denote its difference.
Observe that $d \in \{1,2,3\}$, as otherwise $|B| \leq \frac{n}{4}$, in contradiction to the bound given in~\eqref{eq:|A''|} on the size of $A'' \subseteq B$.
We consider below each of the three possible values of $d$ separately.
\begin{enumerate}
  \item $d=1$: In this case $B = [a,a+\frac{n}{3}-1]$ for some $a \in[1,\frac{2n}{3}+1]$. It suffices to show that $a \geq (\frac{2}{3}-\frac{\eps}{2}) \cdot n$, as this implies that $|A'' \setminus [\frac{2n}{3}+1,n]| \leq \frac{\eps n}{2}+1$, which using~\eqref{eq:|AminusA''|}, yields that
    \[ \Big |A  \setminus \Big[\frac{2n}{3}+1,n \Big] \Big | \leq \Big (\frac{3\eps n}{8} +1 \Big ) + \Big (\frac{\eps n}{2}+1 \Big ) \leq \eps \cdot n,\] as desired.

    By $A'' \subseteq B$, we have that $A''+A'' \subseteq B+B = [2a,2a+\frac{2n}{3}-2]$.
    Recall that $|A''| \geq (\frac{1}{3}-\frac{\eps}{2} ) \cdot n$, hence by Claim~\ref{claim:Freiman} it follows that
    $|A''+A''| \geq 2|A''|-1 \geq (\frac{2}{3}-\eps) \cdot n-1$.
We conclude that $A''$ includes all but at most $\frac{\eps n}{2}$ of the elements of $B$, and that $A''+A''$ includes all but at most $\eps n$ of the elements of $B+B$.
    Since $A''$ is sum-free, it follows that the intervals $B$ and $B+B$ intersect at no more than $\frac{\eps n}{2}+ \eps n= \frac{3\eps n}{2}$ elements, implying that
$a+\frac{n}{3} \leq 2a + \frac{3\eps n}{2}$, thus $a \geq (\frac{1}{3} - \frac{3\eps}{2}) \cdot n$.

We next use the fact that $2n+1 \notin 4A''$, which implies that $A''$ is disjoint from
\[(2n+1) - 3A'' \subseteq (2n+1)-3B = [n-3a+4,2n-3a+1].\]
By Claim~\ref{claim:Freiman} and~\eqref{eq:|A''|}, $|3A''| \geq 3|A''|-2 \geq  (1-\frac{3\eps}{2}) \cdot n-2$.
Therefore, the set $(2n+1)-3A''$ includes all but at most $\frac{3 \eps n}{2}$ of the elements of $(2n+1)-3B$.
This implies that the intervals $B$ and $(2n+1)-3B$ intersect at no more than $\frac{\eps n}{2} + \frac{3\eps n}{2} = 2\eps n$ elements, hence either \[a+\frac{n}{3} \leq n-3a+4 +2\eps n~~~\mbox{or}~~~2n-3a+2 \leq a + 2\eps n.\]
The first possibility is ruled out using $a \geq (\frac{1}{3} - \frac{3\eps}{2}) \cdot n$, so we derive that $a \geq (\frac{1}{2} - \frac{\eps}{2}) \cdot n$.

Finally, we use the fact that $2n+1 \notin 3A''$, which implies that $A''$ is disjoint from
\[(2n+1) - 2A'' \subseteq (2n+1)-2B = \Big [ \frac{4n}{3}-2a+3,2n-2a+1 \Big].\]
Recalling that $|A''+A''| \geq (\frac{2}{3}-\eps) \cdot n-1$, the set $(2n+1)-2A''$ includes all but at most $\eps n$ of the elements of $(2n+1)-2B$.
It follows that the intervals $B$ and $(2n+1)-2B$ intersect at no more than $\frac{\eps n}{2} + \eps n = \frac{3\eps n}{2}$ elements, hence either
\[a+\frac{n}{3} \leq \frac{4n}{3}-2a+3 + \frac{3\eps n}{2}~~~\mbox{or}~~~2n-2a+2 \leq  a+ \frac{3\eps n}{2}.\]
The first possibility is ruled out using $a \geq (\frac{1}{2} - \frac{\eps}{2}) \cdot n$, so we derive that $a \geq (\frac{2}{3} - \frac{\eps}{2}) \cdot n$, and we are done.

  \item $d=2$: In this case the elements of $A''$ are all even or all odd. We show that this is impossible.
In the former alternative, the set $D_0 = \{k \mid 2k \in A''\}$ is a sum-free subset of $[\frac{n}{2}]$. However, the largest size of a sum-free subset of $[n']$ is $\lceil \frac{n'}{2} \rceil$, implying $|A''| = |D_0| \leq \frac{n}{4}$, in contradiction to~\eqref{eq:|A''|}.
For the latter alternative, where the elements of $A''$ are all odd, consider the set $D_1 = \{k \mid 2k-1 \in A'' \} \subseteq [n']$ where $n' = \frac{n}{2}$ (recall Remark~\ref{remark:divisibility}), which by~\eqref{eq:|A''|} satisfies
\[|D_1| = |A''| \geq \Big(\frac{1}{3}-\frac{\eps}{2} \Big) \cdot n = \Big(\frac{2}{3}-\eps \Big) \cdot n'.\]
Observe that the fact that $2n+1 \notin 3A''$ implies that $n+2 \notin 3D_1$, that is, $2n'+2 \notin 3D_1$. By Lemma~\ref{lemma:stability_3A}, all but at most $3\eps n'$ of the elements of $D_1$ are in $[\frac{2n'}{3}]$, and thus all but at most $\frac{3\eps n}{2}$ of the elements of $A''$ are in $[\frac{2n}{3}]$.
Let $E$ denote the set of odd integers in $[\frac{2n}{3}]$ and note that $|E| = \frac{n}{3}$.
Using~\eqref{eq:|A''|}, it follows that there exists a set $F \subseteq E$ satisfying $E \setminus F \subseteq A''$ and $|F| \leq \frac{\eps n}{2} + \frac{3\eps n}{2} = 2\eps n$. To obtain a contradiction we show that $2n+1 \in 5A''$.
Indeed, for every odd integers $x_1,x_2,x_3,x_4 \in [\frac{n}{3}+1,\frac{n}{2}]$ there exists $x_5 \in E$ such that $\sum_{i=1}^{5}{x_i} = 2n+1$, hence there are at least $(\frac{n}{12})^4$ $5$-tuples of elements of $E$ with sum $2n+1$. However, at most $5|F| \cdot (\frac{n}{3})^3 < \eps n^4$ of them involve elements of $F$, so assuming that $\eps$ is sufficiently small, there must exist a $5$-tuple of elements of $E \setminus F \subseteq A''$ whose sum is $2n+1$, as desired.

  \item $d=3$: In this case all the elements of $A''$ are congruent to $r$ modulo $3$ for some $r \in \{0,1,2\}$. We show that this is impossible.
  If $r=0$ then the set $\{k \mid 3k \in A''\}$ is a sum-free subset of $[\frac{n}{3}]$, hence $|A''| \leq \frac{n}{6}$, in contradiction to~\eqref{eq:|A''|}. So assume that $r \in \{1,2\}$.
  Denote by $[n]_r$ the set of integers in $[n]$ which are congruent to $r$ modulo $3$, and notice, using~\eqref{eq:|A''|}, that
  there exists a set $F \subseteq [n]_r$ satisfying  $[n]_r \setminus F \subseteq A''$ and $|F| \leq \frac{\eps n}{2}$.
  For the given $n$ and $r \in \{1,2\}$, let $t$ be the unique integer in $\{3,4,5\}$ satisfying $t \cdot r = 2n+1~(\mod~3)$.
To obtain a contradiction we show that $2n+1 \in tA''$.
By the definition of $t$, for every $x_1,\ldots,x_{t-1} \in [n]_r \cap [\frac{n}{t-1}+1,\frac{2n}{t-1}]$ there exists $x_t \in [n]_r$ such that $\sum_{i=1}^{t}{x_i} = 2n+1$,
hence there are at least $(\frac{n}{3(t-1)})^{t-1}$ $t$-tuples of elements of $[n]_r$ with sum $2n+1$. However, at most $t|F| \cdot (\frac{n}{3})^{t-2} < \eps n^{t-1}$ of them involve elements of $F$, so assuming that $\eps$ is sufficiently small, there must exist a $t$-tuple of elements of $[n]_r \setminus F \subseteq A''$ whose sum is $2n+1$, as desired.
\end{enumerate}
The proof is completed.
\end{proof}

Theorem~\ref{thm:stability} combined with a result of~\cite{BaloghMS15} on counting independent sets in hypergraphs (see Section~\ref{sec:independent}) gives us the following corollary.

\begin{corollary}\label{cor:far_from_extremal}
For every $\gamma >0$, there are $o(2^{n/3})$ sum-free sets $A \subseteq [n]$ satisfying
\[2n+1 \notin (3A) \cup (4A) \cup (5A) \mbox{~~and~~} \Big |A \setminus \Big [\frac{2n}{3}+1,n \Big ] \Big| \geq \gamma n.\]
\end{corollary}

\begin{proof}
Let $\calH = (\calH_n)_{n \in \N}$ be a sequence of $5$-uniform hypergraphs defined as follows. For every $n$, $\calH_n$ is the hypergraph on the vertex set $[n]$ whose hyperedges are all $5$-subsets of $[n]$ that contain at least one of the sets in $\calF^{(3)}_n \cup \calF^{(4)}_n \cup \calF^{(5)}_n$.
Observe that $e(\calH_n) = \Theta(n^4)$ and that $\Delta_1(\calH_n) = \Theta(n^3)$, $\Delta_2(\calH_n) = \Theta(n^2)$, $\Delta_3(\calH_n) = \Theta(n^2)$, $\Delta_4(\calH_n) = \Theta(n)$, and $\Delta_5(\calH_n) = 1$.
Hence, for every sufficiently large $n$, every $\ell \in [5]$, some constant $c>0$, and $p_n = 1/\sqrt{n}$, we have
\[ \Delta_\ell(\calH_n) \leq c \cdot p_n^{\ell-1} \cdot \frac{e(\calH_n)}{v(\calH_n)}.\]
Since every sum-free set $A \subseteq [n]$ satisfying $2n+1 \notin (3A) \cup (4A) \cup (5A)$ forms an independent set in $\calH_n$, it suffices to bound from above, for a given $\gamma >0$, the number of sets $I \in \calI(\calH_n)$ satisfying $|I \setminus [\frac{2n}{3}+1,n]| \geq \gamma n$.

To this end, let us show that the sequence of hypergraphs $\calH$ is $(\frac{1}{3},\calB)$-stable, where $\calB = (\calB_n)_{n \in \N}$ is defined by $\calB_n = \{ [\frac{2n}{3}+1,n] \}$ (recall Definition~\ref{def:dense_stable}, Item~\ref{itm:stable_def}). Let $\eps >0$ be a constant. By Theorem~\ref{thm:stability}, for some $\delta>0$ and every sufficiently large integer $n$, every set $A \subseteq [n]$ of size $|A| \geq (\frac{1}{3}-\delta) \cdot n$ contains at least $\delta \cdot n^{k-1}$ sets from $\calF^{(k)}_n$ for some $k \in \{3,4,5\}$ or satisfies $|A \setminus [\frac{2n}{3}+1,n]| \leq \eps \cdot n$. In the former case we have, using $e(\calH_n) = \Theta(n^4)$, that $e(\calH_n[A]) \geq \delta' \cdot e(\calH_n)$ for some $\delta' >0$, as required.

Now, we can apply Item~\ref{itm:BMS_2} of Theorem~\ref{thm:BMS_items} to obtain that for every $\gamma >0$ there exist $\beta >0$ and $C>0$ such that for every sufficiently large $n$ and every $m \geq C p_n v(\calH_n) = C \sqrt{n}$, there are at most
$(1-\beta)^m \cdot {\binom {n/3} {m}}$
sets $I \in \calI(\calH_n,m)$ satisfying $|I \setminus [\frac{2n}{3}+1,n]| \geq \gamma m$.
In particular, this bound holds on the number of sets $I \in \calI(\calH_n,m)$ satisfying $|I \setminus [\frac{2n}{3}+1,n]| \geq \gamma n$, hence
the total number of sets $I \in \calI(\calH_n)$ satisfying $|I \setminus [\frac{2n}{3}+1,n]| \geq \gamma n$ is at most
\[ \sum_{m=0}^{C \sqrt{n}}{\binom {n} {m}} + \sum_{m=C \sqrt{n}}^{n/3}{ (1-\beta)^m \cdot {\binom {n/3} {m}} }
\leq n^{O(\sqrt{n})} + (1-\beta)^{C \sqrt{n}} \cdot \sum_{m=0}^{n/3}{{\binom {n/3} {m}}} \leq o(2^{n/3}),\]
so we are done.
\end{proof}

\subsubsection{Sets with No Small Elements}\label{sec:n+1_3A}

Our goal in this section is to prove an upper bound on the number of sets $A \subseteq [n]$, all of whose elements are greater than $\frac{n}{2}$, that satisfy $2n+1 \notin 3 A$. To do so, we prove the following theorem, which yields the required bound as an easy corollary. (Recall Remark~\ref{remark:divisibility}.)

\begin{theorem}\label{thm:Grn/2}
There are $O(2^{2n/3})$ sets $A \subseteq [n]$ satisfying $n+1 \notin 3A$.
\end{theorem}

\begin{corollary}\label{cor:Larger_n/2}
There are $O(2^{n/3})$ sets $A \subseteq [\frac{n}{2}+1,n]$ satisfying $2n+1 \notin 3A$.
\end{corollary}

\begin{proof}[ of Corollary~\ref{cor:Larger_n/2}]
Map every set $A \subseteq [\frac{n}{2}+1,n]$ satisfying $2n+1 \notin 3A$ to the set $A' = \{ x- \frac{n}{2} \mid x \in A\} \subseteq [\frac{n}{2}]$, which satisfies $\frac{n}{2}+1 \notin 3A'$. By Theorem~\ref{thm:Grn/2}, the number of possible distinct sets $A'$ is $O(2^{n/3})$. Since the mapping is injective, the corollary follows.
\end{proof}

The proof of Theorem~\ref{thm:Grn/2} employs a counting technique due to~\cite{AlonRefine14}.
For every set $S \subseteq [\frac{n}{3}]$ we consider all sets $A \subseteq [n]$ with $n+1 \notin 3A$ such that $S = A \cap [\frac{n}{3}]$.
The following two claims provide upper bounds on the number of sets $A$ associated with a given $S$. The first is particularly useful for sets $S$ with a large sumset $S+S$, and the second, whose proof uses Janson's inequality, is useful for sets $S$ whose elements are, in average, significantly smaller than $\frac{n}{3}$.

\begin{claim}\label{claim:n/3-|S+S|}
For every integer $n$ and a set $S \subseteq [\frac{n}{3}]$, the number of sets $A \subseteq [n]$ such that $n+1 \notin 3A$ and $S = A \cap [\frac{n}{3}]$ is at most $2^{2n/3-|S+S|}$.
\end{claim}

\begin{proof}
Fix a set $S \subseteq [\frac{n}{3}]$. In order to specify a set $A \subseteq [n]$ satisfying $S = A \cap [\frac{n}{3}]$ one has to specify the elements of $A \cap [\frac{n}{3}+1,n]$. However, the assumption $n+1 \notin 3A$ implies that the $|S+S|$ elements of the set $(n+1)-(S+S) \subseteq [\frac{n}{3}+1,n]$ do not belong to $A$. This implies that the number of sets $A \subseteq [n]$ such that $n+1 \notin 3A$ and $S = A \cap [\frac{n}{3}]$ is at most $2^{2n/3-|S+S|}$, as required.
\end{proof}

\begin{claim}\label{claim:1/2_prob}
For every integers $\ell$, $k$, and $n$, the following holds.
Let $S \subseteq [\frac{n}{3}]$ be a set with $|S|=\ell \geq 1$ and $\sum_{a \in S}{(\frac{n}{3}-a)} = k$.
Then, the number of sets $A \subseteq [n]$ such that $n+1 \notin 3A$ and $S = A \cap [\frac{n}{3}]$ is at most $e^{-k/(32\ell)} \cdot 2^{2n/3}$.
\end{claim}

\begin{proof}
Fix a set $S \subseteq [\frac{n}{3}]$ with $|S|=\ell$ and $\sum_{a \in S}{(\frac{n}{3}-a)} = k$.
To prove the claim, we apply Janson's inequality (Lemma~\ref{lemma:Janson}) with $X = [\frac{n}{3}+1,n]$ and $p = 1/2$ as follows.
Let $\{B_i\}_{i \in J}$ be the collection of all sets $\{x,y\} \subseteq X$ for which $x+y+a = n+1$ for some $a \in S$.
Observe that every $a \in S$ is associated with $\lceil \frac{1}{2}(\frac{n}{3}-a) \rceil$ such sets (one of which might be of size $1$), consisting of the elements of the interval $[\frac{n}{3}+1,\frac{2n}{3}-a] \subseteq X$. This implies that $\frac{k}{2} \leq |J| \leq k$.
Let $R$ denote a random subset of $X$, where every element of $X$ is included in $R$ independently with probability $p$, and notice that if the set $A = S \cup R$ satisfies $n+1 \notin 3A$ then no $B_i$ for $i \in J$ is contained in $R$.
To bound the probability of this event, consider the quantities $\mu$ and $\Delta$ from Lemma~\ref{lemma:Janson}, and observe that $\mu \geq p^2 \cdot |J| \geq \frac{k}{8}$. Further, every set $B_i$ of size $2$ (respectively $1$) intersects at most $2\ell$ (respectively $\ell$) of the other sets, implying that $\Delta \leq p^3 \cdot 2\ell \cdot |J| \leq \frac{\ell k}{4}$. By Janson's inequality, the probability that no $B_i$ is contained in $R$ is at most $\max (e^{-\mu/2},e^{-\mu^2/(2\Delta)})$, which is bounded from above by $e^{-k/(32\ell)}$. This yields an upper bound of $e^{-k/(32\ell)} \cdot 2^{2n/3}$ on the number of sets $A \subseteq [n]$ such that $n+1 \notin 3A$ and $S = A \cap [\frac{n}{3}]$.
\end{proof}

Now we are ready to prove Theorem~\ref{thm:Grn/2}, whose proof uses the above claims and the bounds given in Section~\ref{sec:sumset} on the number of sets of integers with a bounded-size sumset.

\begin{proof}[ of Theorem~\ref{thm:Grn/2}]
Fix a sufficiently small constant $\delta >0$.
For a set $A \subseteq [n]$ denote $S_A = A \cap [\frac{n}{3}]$.
For integers $k$ and $\ell$, let $\calS(k,\ell)$ denote the collection of all sets $S \subseteq [\frac{n}{3}]$ such that $|S| = \ell$ and $\sum_{a \in S}{(\frac{n}{3}-a)} = k$.
Our goal is to bound the size of the collection $\calC$ of all sets $A \subseteq [n]$ satisfying $n+1 \notin 3A$.

We first count the sets $A \in \calC$ for which $S_A \in \calS(k,\ell)$ for $k$ and $\ell$ satisfying $k \geq \ell^2/\delta$.
Since $2^{2n/3}$ subsets of $[n]$ are associated with $S_A = \emptyset$, it can be assumed that $\ell \geq 1$.
For given $k$ and $\ell$, Lemma~\ref{lemma:partitions_basic} implies that
\begin{eqnarray}\label{eq:|S(k,l)|}
|\calS(k,\ell)| \leq \Big ( \frac{e^2 k}{\ell^2}\Big )^\ell.
\end{eqnarray}
By Claim~\ref{claim:1/2_prob}, at most $e^{-k/(32\ell)} \cdot 2^{2n/3}$ sets $A \in \calC$ are associated with the same $S_A \in \calS(k,\ell)$. Hence, for a given $\ell$,
\begin{eqnarray}\label{eq:num_A}
| \{ A \in \calC ~|~ S_A \in \calS(k,\ell) ~\mbox{for some}~ k \geq \ell^2 / \delta \} |
\leq \sum_{k=\ell^2/\delta}^{\infty}{ \Big ( \frac{e^2 k}{\ell^2}\Big )^\ell \cdot e^{-k/(32\ell)} \cdot 2^{2n/3} }.
\end{eqnarray}
To bound the above, observe that the function $g$ defined by $g(x) = ( \frac{e^2 x}{\ell^2})^\ell \cdot e^{-x/(32\ell)}$ satisfies for every $x \geq \ell^2/\delta$,
\[ \frac{g(x+32\ell)}{g(x)} = \Big ( 1+ \frac{32\ell}{x} \Big )^\ell \cdot e^{-1}  \leq  e^{32\ell^2/x} \cdot e^{-1} \leq e^{32 \delta -1} \leq \frac{1}{2},\]
assuming that $\delta$ is sufficiently small.
By $g(\ell^2/\delta) = ( \frac{e^2}{\delta})^\ell \cdot e^{-\ell/(32\delta)}$ and the fact that $g$ is decreasing on $[0,\infty)$, it follows that~\eqref{eq:num_A} is bounded from above by
\[ 64 \ell \cdot \Big (\frac{e^2}{\delta} \Big)^\ell \cdot e^{-\ell/(32\delta)} \cdot 2^{2n/3} \leq e^{-\ell} \cdot 2^{2n/3},\]
where we again use the fact that $\delta$ is sufficiently small.
Summing over all integers $\ell$ we obtain a bound of $O(2^{2n/3})$.

We next count the sets $A \in \calC$ for which $S_A \in \calS(k,\ell)$ for $k$ and $\ell$ satisfying $k < \ell^2/\delta$.
To do so, we consider the following two cases defined by the size of the sumset $S_A+S_A$.
Note that one can assume here, whenever needed, that $\ell$ is sufficiently large, as for a constant $\ell$ and for $k < \ell^2/\delta$ there is only a constant number of sets $S_A \in \calS(k,\ell)$ and they correspond to $O(2^{2n/3})$ sets $A$.
\begin{enumerate}
  \item Consider the sets $A \in \calC$ with $S_A \in \calS(k,\ell)$ satisfying $|S_A+S_A| \geq \ell/\delta$.
Combining~\eqref{eq:|S(k,l)|} with Claim~\ref{claim:n/3-|S+S|}, the number of these sets for a given $\ell$ is at most
\[ \sum_{k=1}^{\ell^2/\delta}{ \Big ( \frac{e^2 k}{\ell^2}\Big )^\ell \cdot 2^{2n/3 -\ell/\delta} } \leq
\frac{\ell^2}{\delta} \cdot \Big ( \frac{e^2}{\delta} \Big )^\ell \cdot 2^{2n/3 -\ell/\delta} \leq
e^{-\ell} \cdot 2^{2n/3},\]
where we have used $k < \ell^2/\delta$ and that $\delta$ is sufficiently small.
Summing over all integers $\ell$ we obtain a bound of $O(2^{2n/3})$.

  \item Consider the sets $A \in \calC$ with $S_A \in \calS(k,\ell)$ satisfying $|S_A+S_A| \leq \ell/\delta$, and recall that
  by Claim~\ref{claim:Freiman} we have $|S_A+S_A| \geq 2\ell-1$.
For such a set $S_A$ let $\lambda \in [2-\delta,1/\delta]$ be a real number satisfying
$\lambda \ell \leq |S_A+S_A| \leq (1+\delta) \cdot \lambda  \ell$.
By Theorem~\ref{thm:partitions_GM}, for a sufficiently large $\ell$, the number of such sets $S_A$ is at most
\[2^{\delta \ell} \cdot {\binom {(1+\delta)\lambda \ell/2}{\ell}} \cdot k ^{(1+\delta)\lambda +\delta} \leq 2^{O(\delta \ell)} \cdot {\binom {(1+\delta)\lambda \ell/2}{\ell}},\]
where for the inequality we have used $k < \ell^2/\delta$, $\lambda \leq 1/\delta$, and the assumption that $\delta$ is sufficiently small.
By Claim~\ref{claim:n/3-|S+S|} we obtain that, for given $\ell$ and $\lambda$, the number of sets $A$ as above is at most
\[ 2^{O(\delta \ell)} \cdot {\binom {(1+\delta)\lambda \ell/2}{\ell}} \cdot 2^{2n/3-\lambda \ell} \leq 2^{O(\delta \ell)} \cdot 2^{(1+\delta)\lambda \ell /2} \cdot 2^{2n/3-\lambda\ell} \leq 2^{-\ell/2} \cdot 2^{2n/3}.\]
Summing over $O(1)$ values of $\lambda$, say $\lambda = (2-\delta)(1+\delta)^j$ for $0 \leq j \leq \frac{2 \ln (1/\delta)}{\delta}$, and over all integers $\ell$, we obtain a bound of $O(2^{2n/3})$.
\end{enumerate}
Summing all the obtained bounds, we get the required bound of $O(2^{2n/3})$, thus the proof is completed.
\end{proof}

\subsubsection{Sets that Include a Small Element and are Almost Contained in $B_n$}\label{sec:easy}

Here we show an easy bound on the number of sum-free sets $A \subseteq [n]$ with $2n+1 \notin 3 A$ that intersect $[\frac{n}{2}]$ and are almost contained in $[\frac{2n}{3}+1,n]$. (Recall Remark~\ref{remark:divisibility}.)

\begin{lemma}\label{lemma:o(2^{n/3})}
For every sufficiently small $\gamma > 0$, there are $o(2^{n/3})$ sum-free sets $A \subseteq [n]$ satisfying $A \cap [\frac{n}{2}] \neq \emptyset$, $2n+1 \notin 3A$, and $|A \setminus [\frac{2n}{3}+1,n]| \leq \gamma n$.
\end{lemma}

\begin{proof}
Fix a set $S \subseteq [\frac{2n}{3}]$ for which there exists $z \in S \cap [\frac{n}{2}]$.
Observe that if $z \in [\frac{n}{6}]$ then there is a collection of at least $n/12$ pairwise disjoint sets $\{x,y\} \subseteq [\frac{2n}{3}+1,n]$ such that $x-y =z$. In addition, if $z \in [\frac{n}{6},\frac{n}{2}]$ then there is a collection of at least $n/12$ pairwise disjoint sets $\{x,y\} \subseteq [\frac{2n}{3}+1,n]$ such that $x+y+z=2n+1$.
In both cases, every sum-free set $A \subseteq [n]$ with $2n+1 \notin 3A$ such that $S = A \setminus [\frac{2n}{3}+1,n]$ does not contain any of these sets, hence the number of such sets $A$ is at most $2^{n/3-n/6} \cdot 3^{n/12} = 2^{n/6} \cdot 3^{n/12}$. Summing over all choices of sets $S \subseteq [\frac{2n}{3}]$ of size at most $\gamma n$ such that $S \cap [\frac{n}{2}] \neq \emptyset$, we get that the total number of sets $A \subseteq [n]$ satisfying $A \cap [\frac{n}{2}] \neq \emptyset$, $2n+1 \notin 3A$, and $|A \setminus [\frac{2n}{3}+1,n]| \leq \gamma n$, is at most
\[\sum_{m=1}^{\gamma n}{{\binom {2n/3} {m}} \cdot 2^{n/6} \cdot 3^{n/12}} \leq 2^{H(3\gamma/2) \cdot 2n/3} \cdot 2^{n/6} \cdot 3^{n/12} \leq o(2^{n/3}),\]
where $H$ stands for the binary entropy function and $\gamma$ is assumed to be sufficiently small.
\end{proof}

\subsubsection{Putting Everything Together}\label{sec:final_proof}

We are finally ready to prove Theorem~\ref{thm:3A4A5A}, which confirms Theorem~\ref{thm:countingIntro}.

\begin{proof}[ of Theorem~\ref{thm:3A4A5A}]
Let $\gamma >0$ be a sufficiently small constant.
By Corollary~\ref{cor:far_from_extremal}, there are $o(2^{n/3})$ sum-free sets $A \subseteq [n]$ satisfying $2n+1 \notin (3A) \cup (4A) \cup (5A)$ and $|A \setminus [\frac{2n}{3}+1,n]| \geq \gamma n$.
Hence, it suffices to bound the number of sum-free sets $A \subseteq [n]$ satisfying $2n+1 \notin 3A$ and $|A \setminus [\frac{2n}{3}+1,n]| < \gamma n$.
By Lemma~\ref{lemma:o(2^{n/3})}, at most $o(2^{n/3})$ of them intersect $[\frac{n}{2}]$, and by Corollary~\ref{cor:Larger_n/2}, at most $O(2^{n/3})$ of them do not. This completes the proof.
\end{proof}

\subsection{On the Tightness of Theorem~\ref{thm:n/3}}\label{sec:3A}

Theorem~\ref{thm:extremalIntro} provides a tight upper bound of $\lfloor \frac{1}{3}(n+1) \rfloor$ on the size of every sum-free set $A \subseteq [n]$ satisfying $2n+1 \notin \sum A$.
In fact, the bound is shown in Theorem~\ref{thm:n/3} even under the weaker assumption $2n+1 \notin (3A) \cup (4A)$.
We claim that Theorem~\ref{thm:n/3} is tight not only because of the bound that it provides, but also in the sense that the assumption $2n+1 \notin (3A) \cup (4A)$ cannot be relaxed to $2n+1 \notin 3A$. To see this, let $n$ be an integer satisfying $n = 2~(\mod~5)$, and consider the set
\begin{equation}\label{eq:Set_1,4}
A = \{ x \in [n] \mid x=1~\mbox{ or }~4 ~(\mod~5)\}.
\end{equation}
Notice that over $\Z_5$ we have $2\{1,4\} = \{0,2,3\}$, implying that $A$ is sum-free. We also have, over $\Z_5$, that $3\{1,4\} = \{0,2,3\}+\{1,4\} = \{1,2,3,4\}$, implying that $2n+1$, which is divisible by $5$, is not in $3A$. It follows that for every integer $n$ such that $n = 2~(\mod~5)$ there exists a sum-free set $A \subseteq [n]$ satisfying $2n+1 \notin 3A$ whose size is $|A| = \lceil \frac{2n}{5} \rceil$. Another example of such a set is given by the set of all integers of $[n]$ that are congruent to $2$ or $3$ modulo $5$. The following theorem shows that these constructions achieve the largest possible size of a set with these properties.

\begin{theorem}\label{thm:2n/5}
For every sufficiently large integer $n$, every sum-free set $A \subseteq [n]$ such that $2n+1 \notin 3A$ satisfies
\[|A| \leq \Big \lceil \frac{2n}{5} \Big \rceil .\]
\end{theorem}

The proof of Theorem~\ref{thm:2n/5} uses a recent characterization of large sum-free subsets of $[n]$ due to Tran~\cite{Tran17} (see Theorem~\ref{thm:Tran}).
We start with the following lemma.

\begin{lemma}\label{lemma:n/3_min_large}
For every integer $n$, every nonempty sum-free set $A \subseteq [n]$ such that $\min(A) > \frac{n}{3}$ and $2n+1 \notin 3A$ satisfies $|A| \leq \lfloor \frac{1}{3}(n+1) \rfloor$.
\end{lemma}

\begin{proof}
Let $A \subseteq [n]$ be a sum-free set with $2n+1 \notin 3A$, and denote $k = \min (A) > \frac{n}{3}$.
Assume first that $k \leq \lfloor \frac{n}{2} \rfloor$.
Consider the collection of pairwise disjoint pairs
\[B_1(k) = \{ (x,x+k) \mid x \in [k,n-k] \},\]
that involves all the elements of $[k,n-k] \cup [2k,n]$ (Note that $k \leq n-k < 2k \leq n$).
Since $A$ is sum-free and includes $k$, for every pair $(x,y) \in B_1(k)$ we have $x \notin A$ or $y \notin A$.
Additionally, consider the collection of pairwise disjoint pairs
\[B_2(k) = \Big\{ (x,2n+1-2x) ~\Big |~ x \in \Big [n-k+1, \Big \lfloor\frac{1}{3}(2n+1) \Big \rfloor \Big ]  \Big\},\]
that involves elements from $[n-k+1,2k-1]$ (Note that $n-k+1 \leq 2k$). Notice that for every pair $(x,y) \in B_2(k)$ we have $2x+y=2n+1$, so by $2n+1 \notin 3A$ we have $x \notin A$ or $y \notin A$.
Since the pairs of $B_1(k) \cup B_2(k)$ are pairwise disjoint, it follows that the size of the set $A \subseteq [k,n]$ satisfies
\begin{eqnarray*}
|A| &\leq& (n-k+1) - |B_1(k)| -|B_2(k)| \\
&=& (n-k+1) -  (n-2k+1) - \Big( \Big\lfloor \frac{1}{3}(2n+1) \Big\rfloor -(n-k) \Big) \\
&=& n- \Big\lfloor \frac{1}{3}(2n+1) \Big\rfloor = \Big \lfloor \frac{1}{3}(n+1) \Big \rfloor,
\end{eqnarray*}
as required.
Assume next that $ k > \lfloor \frac{n}{2} \rfloor$. Here, $A \subseteq [ \lfloor \frac{n}{2} \rfloor +1 ,n]$ and a similar argument implies that
\[|A| \leq \Big\lceil \frac{n}{2} \Big\rceil - \Big |B_2 \Big( \Big\lceil \frac{n}{2} \Big\rceil \Big) \Big| =
\Big\lceil \frac{n}{2} \Big\rceil - \Big ( \Big\lfloor \frac{1}{3}(2n+1) \Big\rfloor - \Big\lfloor \frac{n}{2} \Big\rfloor \Big ) =
n- \Big\lfloor \frac{1}{3}(2n+1) \Big\rfloor = \Big \lfloor \frac{1}{3}(n+1) \Big \rfloor,\]
so we are done.
\end{proof}

\begin{proof}[ of Theorem~\ref{thm:2n/5}]
For a sufficiently large integer $n$, let $A \subseteq [n]$ be a sum-free set such that $2n+1 \notin 3A$.
Apply Theorem~\ref{thm:Tran} with $\eta = \frac{2}{n}$.
Assume by contradiction that $|A| > \lceil \frac{2n}{5} \rceil \geq (\frac{2}{5}-\eta) \cdot n$.
It suffices to prove that $A$ does not satisfy any of the five alternatives in Theorem~\ref{thm:Tran}.

Alternatives~\eqref{itm:1} and~\eqref{itm:2} are not satisfied because at most $\lceil \frac{2n}{5} \rceil$ of the elements of $[n]$ are congruent to $1$ or $4$ modulo $5$, and at most $\lceil \frac{2n}{5} \rceil$ of them are congruent to $2$ or $3$ modulo $5$.

For alternative~\eqref{itm:3}, notice that if all the elements of $A$ are odd then $2n+1 \notin 4A$.
By Theorem~\ref{thm:n/3}, using $2n+1 \notin 3A$, we get that $|A| \leq \lfloor \frac{1}{3}(n+1) \rfloor$, in contradiction.

For alternative~\eqref{itm:4}, use Lemma~\ref{lemma:n/3_min_large} to obtain that if $\min(A) \geq |A|> \frac{n}{3}$ then $|A| \leq \lfloor \frac{1}{3}(n+1) \rfloor$, again in contradiction.

Finally, notice that if alternative~\eqref{itm:5} holds, with $\eta = \frac{2}{n}$, then there exists a constant $d>0$ for which
\[A \subseteq \Big[\frac{n}{5}-d \cdot \sqrt{n},\frac{2n}{5}+d \cdot \sqrt{n} \Big] \cup \Big[\frac{4n}{5}-d \cdot \sqrt{n},n \Big].\]
Denote $B = [\frac{n}{5},\frac{2n}{5}] \cup [\frac{4n}{5},n]$ and notice that the assumption $|A| > \lceil \frac{2n}{5} \rceil$ implies that there exists $D \subseteq B$ of size $|D| \leq 3d \cdot \sqrt{n}$ such that $B \setminus D \subseteq A$. To obtain a contradiction, we show that $2n+1 \in 3A$. Indeed, for every $x,y \in [\frac{4n}{5}+1,\frac{9n}{10}] \subseteq B$ there exists $z \in [\frac{n}{5},\frac{2n}{5}] \subseteq B$ such that $x+y+z = 2n+1$, hence there are at least $(\frac{n}{10})^2$ triples of elements of $B$ with sum $2n+1$. However, at most $3|D| \cdot |B| < 4d \cdot n^{1.5}$ of them involve elements of $D$, so for a sufficiently large $n$ there must exist a triple of elements of $B \setminus D \subseteq A$ whose sum is $2n+1$, and we are done.
\end{proof}

While the bound given in Theorem~\ref{thm:2n/5} is tight for integers $n$ satisfying $n = 2~(\mod~5)$ (see~\eqref{eq:Set_1,4}), it turns out that this is not the case in general, as is shown in the following theorem.

\begin{theorem}\label{thm:(2/5-c)n}
There exists a constant $\eta >0$ such that for every sufficiently large integer $n$ such that $n \neq 2~(\mod~5)$, every sum-free set $A \subseteq [n]$ such that $2n+1 \notin 3A$ satisfies $|A| \leq (\frac{2}{5}-\eta) \cdot n$.
\end{theorem}

\begin{proof}
For a sufficiently large integer $n$ such that $n \neq 2~(\mod~5)$, let $A \subseteq [n]$ be a sum-free set satisfying $2n+1 \notin 3A$.
Apply Theorem~\ref{thm:Tran} with a sufficiently small constant $\eta >0$.
Assume by contradiction that $|A| > (\frac{2}{5}-\eta) \cdot n$.
It suffices to prove that $A$ does not satisfy any of the five alternatives in Theorem~\ref{thm:Tran}.

Assume that alternative~\eqref{itm:1} holds, that is, all the elements of $A$ are congruent to $1$ or $4$ modulo $5$.
Denote by $[n]_r$ the set of integers in $[n]$ which are congruent to $r$ modulo $5$.
Letting $A = A_1 \cup A_4$ where $A_r \subseteq [n]_r$ for $r \in \{1,4\}$, the assumption $|A| > (\frac{2}{5}-\eta) \cdot n$ implies that $|[n]_r \setminus A_r| < \eta \cdot n$ for $r \in \{1,4\}$. To obtain a contradiction, we show that $2n+1 \in 3A$.
Observe that by $n \neq 2~(\mod~5)$ there exist $r_1,r_2,r_3 \in \{1,4\}$ such that $r_1+r_2+r_3=2n+1~(\mod~5)$.
Further, for every $x \in [n]_{r_1} \cap [\frac{n}{2}+1,n]$ and $y \in [n]_{r_2} \cap [\frac{n}{2}+1,n]$ there exists $z \in [n]_{r_3}$ such that $x+y+z=2n+1$, hence there are at least $(\frac{n}{10})^2$ triples in $[n]_{r_1} \times [n]_{r_2} \times [n]_{r_3}$ with sum $2n+1$. However, at most $3 \cdot \eta n \cdot n = 3 \eta n^2$ of them involve elements not in $A$, so for a sufficiently small $\eta$ there must exist a triple of elements of $A$ whose sum is $2n+1$, as required.

Alternative~\eqref{itm:2} is handled similarly to the way alternative~\eqref{itm:1} is, so we omit the details.

For alternative~\eqref{itm:3}, notice that if all the elements of $A$ are odd then $2n+1 \notin 4A$.
By Theorem~\ref{thm:n/3}, using $2n+1 \notin 3A$, we get that $|A| \leq \lfloor \frac{1}{3}(n+1) \rfloor$, in contradiction.

For alternative~\eqref{itm:4}, use Lemma~\ref{lemma:n/3_min_large} to obtain that if $\min(A) \geq |A|> \frac{n}{3}$ then $|A| \leq \lfloor \frac{1}{3}(n+1) \rfloor$, again in contradiction.

Finally, notice that if alternative~\eqref{itm:5} holds then for $d = 200 \eta^{1/2}$ we have
\[A \subseteq \Big[\Big(\frac{1}{5}-d \Big) \cdot n,\Big(\frac{2}{5}+d \Big) \cdot n \Big] \cup \Big[\Big(\frac{4}{5}-d \Big) \cdot n,n \Big].\]
Denote $B = [\frac{n}{5},\frac{2n}{5}] \cup [\frac{4n}{5},n]$ and notice that there exists $D \subseteq B$ of size $|D| \leq 3d n$ such that $B \setminus D \subseteq A$. To obtain a contradiction, we show that $2n+1 \in 3A$.
Recall that there are at least $(\frac{n}{10})^2$ triples of elements of $B$ with sum $2n+1$ (see the proof of Theorem~\ref{thm:2n/5}). However, at most $3|D| \cdot |B| < 4dn^2$ of them involve elements of $D$, so for a sufficiently small $\eta$ there must exist a triple of elements of $B \setminus D \subseteq A$ whose sum is $2n+1$, and we are done.
\end{proof}

\subsection{Other Forbidden Sums}\label{sec:other_forbidden}

Although the current work focuses on sum-free sets $A \subseteq [n]$ that satisfy $2n+1 \notin \sum A$, it is natural to consider other forbidden sums, different from $2n+1$, as well.
We first observe that Theorem~\ref{thm:countingIntro} can be used to obtain a tight estimation for the number of sum-free subsets of $[n]$ with an odd forbidden sum around $2n$.

\begin{corollary}\label{cor:general_odd_even}
For every fixed integer $k \in \Z$, there are $\Theta(2^{n/3})$ sum-free sets $A \subseteq [n]$ satisfying $2(n+k)+1 \notin \sum A$.
\end{corollary}

\begin{proof}
Apply Theorem~\ref{thm:countingIntro} to get that there are $\Theta(2^{(n+k)/3}) = \Theta(2^{n/3})$ sum-free sets $A \subseteq [n+k]$ satisfying $2(n+k)+1 \notin \sum A$.
Considering the subsets of $[n]$ with this property, the estimation is affected by no more than a multiplicative constant factor of $2^k$, so we are done.
\end{proof}

For the forbidden sum $2n$, we prove the following extremal result, whose proof resembles that of Theorem~\ref{thm:extremalIntro}.

\begin{theorem}\label{thm:extremal_2n}
For every integer $n$, the maximum size of a sum-free set $A \subseteq [n]$ such that $2n \notin \sum A$ is $\lfloor \frac{1}{3}(n-1) \rfloor$.
\end{theorem}

\begin{proof}
For the upper bound, let $A \subseteq [n]$ be a sum-free set such that $2n \notin \sum A$. Consider the set $C = A \cup (A+A)$.
By the sum-freeness of $A$ and Claim~\ref{claim:Freiman},
$|C| = |A| + |A + A| \geq 3|A|-1$.
We claim that $|C| \leq n-2$. To see this, notice that $1 \notin A$ and that $n \notin C$, hence $C \subseteq [2,2n-2] \setminus \{n\}$. If $|C| > n-2$ then, by the pigeonhole principle, there exists $2 \leq i \leq n-1$ for which $\{i, 2n-i\} \subseteq C$, implying that $2n \in \sum A$, in contradiction.
It follows that $3|A|-1 \leq |C| \leq n-2$, hence $|A| \leq \lfloor \frac{1}{3}(n-1) \rfloor$ as required.
The set $[\lceil \frac{1}{3}(2n+1) \rceil,n-1]$ implies a matching lower bound.
\end{proof}

We remark that it can be verified that the proof technique of Theorem~\ref{thm:3A4A5A} can be used to obtain a tight bound of $O(2^{n/3})$ on the number of sum-free sets $A \subseteq [n]$ satisfying $2n \notin (3A) \cup (4A) \cup (5A)$.
Combining it with the lower bound that follows from Theorem~\ref{thm:extremal_2n}, we get that there are $\Theta(2^{n/3})$ sum-free sets $A \subseteq [n]$ such that $2n \notin \sum A$ (and, as in Corollary~\ref{cor:general_odd_even}, one can derive such a bound for every even forbidden sum around $2n$).
However, we point out a significant difference between the forbidden sums $2n$ and $2n+1$.
We prove below a bound of $2^{(\frac{1}{3}+o(1)) \cdot n}$ on the number of sets $A \subseteq [n]$ satisfying $2n \notin \sum A$, which holds even without assuming the sum-freeness of the sets. This is in contrast to the forbidden sum $2n+1$ that is avoided by the $2^{\lfloor n/2 \rfloor}$ subsets of $[n]$ that consist of only even integers.

\begin{proposition}\label{prop:2n}
There are $2^{(\frac{1}{3} + o(1)) \cdot n}$ sets $A \subseteq [n]$ satisfying $2n \notin \sum A$.
\end{proposition}
\noindent
The proof uses the following special case of a result of Alon~\cite{Alon87}.

\begin{theorem}[\cite{Alon87}]\label{thm:Alon}
For every $\eps>0$ and every sufficiently large integer $n$, every set $A \subseteq \Z_n$ of size $|A| \geq (\frac{1}{3}+\eps) \cdot n$ contains a set $B \subseteq A$ satisfying $0 < |B| \leq 3$ and $\sum_{b \in B}{b} = 0$ (in $\Z_n$).
\end{theorem}

\begin{proof}[ of Proposition~\ref{prop:2n}]
We first claim that for every $\eps>0$ and every sufficiently large integer $n$, every set $A \subseteq [n]$ of size $|A| \geq (\frac{1}{3}+\eps) \cdot n$ satisfies $2n \in (3A) \cup (4A) \cup (6A)$. Indeed, applying Theorem~\ref{thm:Alon} to the set $A \setminus \{n\}$ considered as a subset of $\Z_n$, it follows that $n \in (2A) \cup (3A)$ or $2n \in 3A$, thus $2n \in (3A) \cup (4A) \cup (6A)$.

We next obtain the following supersaturation statement: For every $\eps >0$ there exists $\delta >0$ such that for every sufficiently large integer $n$, every set $A \subseteq [n]$ of size $|A| \geq (\frac{1}{3}+\eps) \cdot n$ contains at least $\delta \cdot n^{k-1}$ $k$-subsets of $[n]$ with sum $2n$ for some $k \in \{3,4,6\}$. To see it, for a given $\eps >0$ define $\delta = \min( \delta''_3 (\frac{\eps}{6}), \delta''_4 (\frac{\eps}{6}), \delta''_6 (\frac{\eps}{6}))$ where $\delta''_3$, $\delta''_4$, and $\delta''_6$ are as in Corollary~\ref{cor:3A4A_sets}. Assume by contradiction that for a sufficiently large $n$, a set $A \subseteq [n]$ of size $|A| \geq (\frac{1}{3}+\eps) \cdot n$ contains fewer than $\delta \cdot n^{k-1}$ $k$-subsets of $[n]$ with sum $2n$ for every $k \in \{3,4,6\}$. Applying Corollary~\ref{cor:3A4A_sets} three times with $\ell = 2n$ and $k \in \{3,4,6\}$ we obtain that there exist sets $D_1, D_2, D_3 \subseteq A$, each of which is of size at most $\frac{\eps}{6} \cdot n$, such that $2n \notin 3(A \setminus D_1) \cup 4(A \setminus D_2) \cup 6(A \setminus D_3)$. Hence, the set $B = A \setminus (D_1 \cup D_2 \cup D_3)$ has size $|B| \geq (\frac{1}{3}+\frac{\eps}{2}) \cdot n$, and yet satisfies $2n \notin (3B) \cup (4B) \cup (6B)$, in contradiction.

Now, let $\calH = (\calH_n)_{n \in \N}$ be a sequence of $6$-uniform hypergraphs defined as follows. For every $n$, $\calH_n$ is the hypergraph on the vertex set $[n]$ whose hyperedges are all $6$-subsets of $[n]$ that contain $3$, $4$, or $6$ distinct elements of $[n]$ whose sum is $2n$.
It is straightforward to verify that $e(\calH_n) = \Theta(n^5)$, and that $\Delta_1(\calH_n) = \Theta(n^4)$, $\Delta_2(\calH_n) = \Theta(n^3)$, $\Delta_3(\calH_n) = \Theta(n^3)$, $\Delta_4(\calH_n) = \Theta(n^2)$, $\Delta_5(\calH_n) = \Theta(n)$, and $\Delta_6(\calH_n) = 1$.
Hence, for every sufficiently large $n$, every $\ell \in [6]$, some constant $c>0$, and $p_n = 1/\sqrt{n}$, we have
\[ \Delta_\ell(\calH_n) \leq c \cdot p_n^{\ell-1} \cdot \frac{e(\calH_n)}{v(\calH_n)}.\]
Since every set $A \subseteq [n]$ satisfying $2n \notin \sum A$ forms an independent set in $\calH_n$, it suffices to bound from above the size of $\calI(\calH_n)$. By the above supersaturation result, the sequence of hypergraphs $\calH$ is $\frac{1}{3}$-dense (recall Definition~\ref{def:dense_stable}, Item~\ref{itm:dense_def}).
Applying Item~\ref{itm:BMS_1} of Theorem~\ref{thm:BMS_items}, we obtain that for every $\gamma >0$ there exists $C>0$ such that for every sufficiently large $n$ and every $m \geq C p_n v(\calH_n) = C \sqrt{n}$,
\[ |\calI(\calH_n,m)| \leq {\binom {(\frac{1}{3}+\gamma)n} {m}}.\]
It follows that
\[ |\calI(\calH_n)| \leq \sum_{m=0}^{C\sqrt{n}}{{\binom {n}{m}}} + \sum_{m=C\sqrt{n}}^{(\frac{1}{3}+\gamma)n}{ {\binom {(\frac{1}{3}+\gamma)n} {m}} } \leq n^{O(\sqrt{n})}+2^{(\frac{1}{3}+\gamma) \cdot n} \leq 2^{(\frac{1}{3}+2\gamma) \cdot n}.\]
Since the bound holds for every $\gamma >0$, the result follows.
\end{proof}

\section{Symmetric Complete Sum-free Sets in Cyclic Groups}\label{sec:Z_p}

In this section we relate our study of sum-free subsets of $[n]$ with $2n+1$ as a forbidden sum to counting symmetric complete sum-free sets in cyclic groups of prime order (see Section~\ref{sec:additive} for the definitions).

For a prime $p$, an even integer $s \in [\frac{p+3}{4},\frac{p-1}{3}]$, and a set of integers $T \subseteq [0,2t-1]$ where $t = (p-3s+1)/2 \geq 1$, consider the set $S_T \subseteq \Z_p$ defined by
\[ S_T = [p-2s+1,2s-1] \cup \pm (s+T).\]
Observe that $|T|=t$ if and only if $|S_T| = (4s-p-1)+2t = s$, and that $0 \in T$ if and only if $s \in S_T$.
As mentioned earlier, it was shown in~\cite{HavivLsf17} that for every sufficiently large prime $p$, every symmetric complete sum-free subset of $\Z_p$ of even size $s \geq 0.318p$ is, up to an automorphism, of the form $S_T$ for some set of integers $T \subseteq [0,2t-1]$ of size $t$. Moreover, those sets $T$ for which $S_T$ is complete and sum-free were fully characterized in~\cite{HavivLsf17}.
To state the characterization result, we need the following definition.

\begin{definition}\label{def:special}
For an integer $t \geq 1$ we say that a set $T \subseteq [0,2t-1]$ is {\em $t$-special} if it satisfies
\begin{enumerate}
  \item $|T| = t$,
  \item $2t-1 \notin 3T$, and
  \item $[0,2t-1+\min(T)] \setminus (2t-1-T) \subseteq T+T$.
\end{enumerate}
Note that the addition is over the integers.
\end{definition}

\begin{theorem}[\cite{HavivLsf17}]\label{thm:Final_Char}
For every sufficiently large prime $p$ and every even integer $s \in [0.318p,\frac{p-1}{3}]$, the following holds.
The symmetric complete sum-free subsets of $\Z_p$ of size $s$ are precisely all the dilations of the sets $S_T$ for $t$-special sets $T \subseteq [0,2t-1]$ where $t=(p-3s+1)/2$.
\end{theorem}

The above theorem reduces the challenge of counting the symmetric complete sum-free subsets of $\Z_p$ of a sufficiently large size to counting the $t$-special sets for an appropriate $t$ (There is an exact multiplicative gap of $(p-1)/2$ between the two; see~\cite[Theorem~1.2]{HavivLsf17}). As an application of Theorem~\ref{thm:countingIntro}, we prove below the following tight estimation for the number of $t$-special sets $T \subseteq [0,2t-1]$ that satisfy $0 \in T$.

\begin{theorem}\label{thm:t-special}
There are $\Theta(2^{t/3})$ $t$-special sets $T \subseteq [0,2t-1]$ satisfying $0 \in T$.
\end{theorem}

This easily implies, for a sufficiently large $s$, a tight estimation for the number of symmetric complete sum-free sets $S_T \subseteq \Z_p$ of size $s$ that satisfy $s \in S_T$, confirming Theorem~\ref{thm:S_T_with_sIntro}.

\begin{proof}[ of Theorem~\ref{thm:S_T_with_sIntro}]
Let $p$ be a sufficiently large prime and let $s \in [0.318p,\frac{p-1}{3}]$ be an even integer. Denote $t=(p-3s+1)/2$.
By Theorem~\ref{thm:Final_Char}, the number of symmetric complete sum-free sets $S_T \subseteq \Z_p$ of size $s$ that satisfy $s \in S_T$ is equal to the number of $t$-special sets that satisfy $0 \in T$.
By Theorem~\ref{thm:t-special}, the latter is equal to $\Theta(2^{t/3}) = \Theta(2^{(p-3s)/6})$, so we are done.
\end{proof}

\subsection{Counting Special Sets -- Proof of Theorem~\ref{thm:t-special}}

We start with the following simple claim that shows that for a set that includes $0$, the first and second conditions in Definition~\ref{def:special} imply the third.

\begin{claim}\label{claim:T_special}
For every integer $t \geq 1$ and a set $T \subseteq [0,2t-1]$ satisfying $0 \in T$, $|T|=t$, and $2t-1 \notin 3T$, the following holds.
\begin{enumerate}
  \item\label{itm:T_1} For every $\ell \in [0,t-1]$ exactly one of the elements $\ell$ and $2t-1-\ell$ belongs to $T$.
  \item\label{itm:T_2} For every $\ell_1,\ell_2 \in T$, if $\ell_1+\ell_2 \in [0,2t-1]$ then $\ell_1+\ell_2 \in T$.
  \item\label{itm:T_3} $T$ is $t$-special.
\end{enumerate}
\end{claim}

\begin{proof}
Assume that a set $T \subseteq [0,2t-1]$ satisfies $0 \in T$, $|T|=t$, and $2t-1 \notin 3T$.
\begin{enumerate}
  \item Let $\ell \in [0,t-1]$. The elements $\ell$ and $2t-1-\ell$ cannot both belong to $T$ as their sum, together with $0 \in T$, is $2t-1$, which does not belong to $3T$. In addition, if $T$ does not include one of $\ell$ and $2t-1-\ell$ for some $\ell \in [0,t-1]$ then its size cannot reach $t$. This completes the proof of the first item.
  \item For $\ell_1,\ell_2 \in T$ assume that $\ell_1+\ell_2 \in [0,2t-1]$. By $2t-1 \notin 3T$, it follows that $(2t-1)-(\ell_1+\ell_2) \notin T$, hence by Item~\ref{itm:T_1}, $\ell_1+\ell_2 \in T$.
  \item To prove that $T$ is $t$-special it suffices to show that $[0,2t-1] \setminus (2t-1-T) \subseteq T+T$. Indeed, every $\ell \in [0,2t-1]$ for which $2t-1-\ell \notin T$ satisfies, by Item~\ref{itm:T_1}, $\ell = \ell+0 \in T+T$, and we are done.
\end{enumerate}
\end{proof}

Consider the following notion of sets that are closed under addition.

\begin{definition}
For an integer $n$, we say that a set $A \subseteq [n]$ is {\em closed under addition} if for every $x,y \in A$ such that $x+y \in [n]$ we have $x+y \in A$.
Let $\calA_n$ denote the collection of all sets $A \subseteq [n]$ that are closed under addition and satisfy $2n+1 \notin 3 A$.
\end{definition}

In the following lemma, we relate counting special sets that include $0$ to counting sets that are closed under addition.

\begin{lemma}\label{lemma:special-closed}
For an integer $t \geq 1$, let $\calT_t$ be the collection of all $t$-special sets $T \subseteq [0,2t-1]$ satisfying $0 \in T$. Then, $|\calT_t| = |\calA_{t-1}|$.
\end{lemma}

\begin{proof}
For an integer $t \geq 1$, consider the function $g$ that maps every set $T \in \calT_t$ to the set \[g(T) = T \cap [t-1].\]
To prove the lemma, it suffices to show that $g$ is a bijection from $\calT_t$ to $\calA_{t-1}$.
We first observe that for every $T \in \calT_t$ we have $g(T) \in \calA_{t-1}$.
Indeed, for $T \in \calT_t$ we have $2t-1 \notin 3T$ and thus, by $g(T) \subseteq T$, we also have $2t-1 \notin 3g(T)$. Further, using Item~\ref{itm:T_2} of Claim~\ref{claim:T_special}, it follows that $g(T)$ is closed under addition as a subset of $[t-1]$, hence $g(T) \in \calA_{t-1}$.

The function $g$ is injective by Item~\ref{itm:T_1} of Claim~\ref{claim:T_special} and the fact that $0 \in T$ for every $T \in \calT_t$.
To prove that $g$ is surjective, take a set $A \in \calA_{t-1}$ and define
\[T = \{0\} \cup A \cup \Big \{ (2t-1)-\ell ~\Big |~ \ell \in [t-1] \setminus A \Big \}.\]
We clearly have $g(T)=A$, $0 \in T$, and $|T|=t$. By Item~\ref{itm:T_3} of Claim~\ref{claim:T_special}, to prove that $T \in \calT_t$ it suffices to show that $2t-1 \notin 3T$.
Assume by contradiction that there are $x_1,x_2,x_3 \in T$ such that $x_1+x_2+x_3 =2t-1$.
Observe, using $2t-1 \notin 3A$, that exactly one of $x_1,x_2,x_3$, say $x_3$, is in $[t,2t-1]$.
Since $2t-1 \notin T$, it follows that $x_3 \in [t,2t-2]$ and thus $x_1 + x_2 \in [t-1]$.
Using the fact that $A$ is closed under addition, we get that $x_1+x_2 \in T$. However, $x_3 = (2t-1)-(x_1+x_2) \in T$, in contradiction to the fact that $T$, by definition, contains no two elements whose sum is $2t-1$, and we are done.
\end{proof}

We turn to estimate the size of $\calA_n$.
To do so, we need the following lemma.

\begin{lemma}\label{lemma:closed-sumfree}
For an integer $n \geq 1$, let $\calD_n$ be the collection of all sum-free sets $D \subseteq [n]$ that satisfy $2n+1 \notin \sum D$. Then, $|\calA_n| \leq |\calD_n|$.
\end{lemma}

\begin{proof}
For an integer $n \geq 1$, consider the function $f$ that maps every set $A \in \calA_n$ to the set $f(A)$ defined by the following process: Go over the elements of $A$ from the largest one to the smallest one, and exclude every element of $A$ that forms a sum of two smaller elements in the set. To prove the lemma, it suffices to show that $f$ is an injective function from $\calA_n$ to $\calD_n$.

We first show that for every $A \in \calA_n$ we have $f(A) \in \calD_n$. Let $A \in \calA_n$. It follows from the definition that $f(A)$ is a sum-free subset of $[n]$. Assume by contradiction that $2n+1 \in \sum f(A)$, that is, there exist $k \geq 3$ and $x_1, \ldots, x_k \in f(A)$ such that $\sum_{i=1}^{k}{x_i} = 2n+1$. We claim that the integers $x_1, \ldots, x_k$ can be partitioned into $3$ sets, each of which is of sum at most $n$.
To see this, start with the singleton partition and repeatedly combine pairs of sets whose joint sum is at most $n$. This process must terminate with exactly $3$ sets, since there are no $4$ integers with total sum at most $2n+1$ such that the sum of every two of them is larger than $n$.
Since $A \subseteq [n]$ is closed under addition and contains $f(A)$, it follows that $2n+1 \in 3A$, in contradiction.

To complete the proof, we show that $f$ is injective. For distinct sets $A_1, A_2 \in \calA_n$, let $x$ be the smallest integer that belongs to exactly one of them, and assume without loss of generality that $x \in A_1$ and $x \notin A_2$. Since $A_2$ is closed under addition, $x$ is not a sum of two elements of $A_2$, and by the minimality of $x$, it is not a sum of two elements of $A_1$ as well.
Therefore, $x \in f(A_1)$ and $x \notin f(A_2)$, hence $f(A_1) \neq f(A_2)$.
\end{proof}

\begin{corollary}\label{cor:A_n}
$|\calA_n| = \Theta(2^{n/3})$.
\end{corollary}

\begin{proof}
For the upper bound, combine Lemma~\ref{lemma:closed-sumfree} with Theorem~\ref{thm:countingIntro}.
For the lower bound, consider all subsets of the set $[ \lceil \frac{2}{3}(n+1) \rceil ,n]$ to obtain that $|\calA_n| \geq 2^{\lfloor \frac{1}{3}(n+1) \rfloor}$.
\end{proof}
\noindent
Finally, to derive Theorem~\ref{thm:t-special}, combine Lemma~\ref{lemma:special-closed} with Corollary~\ref{cor:A_n}.

\section*{Acknowledgments}
We are deeply indebted to Wojciech Samotij for helpful discussions and suggestions.
We also thank the anonymous referees for their valuable comments that improved the presentation of the paper.

\bibliographystyle{abbrv}
\bibliography{counting_sf}

\end{document}